\documentclass[11pt]{article} 

\usepackage[utf8]{inputenc} 

\usepackage{geometry} 
\usepackage{graphicx} 

\usepackage{booktabs} 
\usepackage{array} 
\usepackage{paralist} 
\usepackage{verbatim} 
\usepackage{subfig} 
\usepackage{amsfonts}
\usepackage{amsthm}
\usepackage{amsmath}
\usepackage{amssymb}

\usepackage{fancyhdr} 
\usepackage{sectsty}
\allsectionsfont{\sffamily\mdseries\upshape} 

\newcommand{\HC}{{\textrm{HC}}}

\newcommand{\css}{{\textrm{css}}}
\newcommand{\I}{{\mathcal{I}}}
\newcommand{\J}{{\mathcal{J}}}

\theoremstyle{definition}
\newtheorem{theorem}{Theorem}
\newtheorem{corollary}[theorem]{Corollary}

\newtheorem{lemma}[theorem]{Lemma}

\newtheorem{definition}[theorem]{Definition}
\newtheorem{remark}[theorem]{Remark}

\makeatletter
\def\blfootnote{\xdef\@thefnmark{}\@footnotetext}
\makeatother

\begin{document}


\author{Saharon Shelah\!\! \thanks{
This material is based in part upon work supported by the National
Science Foundation under Grant No. 136974 and  by European Research Council grant 338821. Paper 1141 on Shelah's list.}\\
Douglas Ulrich \!\!\
	\thanks{Partially supported
by Laskowski's NSF grant DMS-1308546.}}
	\title{Torsion-Free Abelian Groups are Consistently $a \Delta^1_2$-complete}
	\date{\today} 
%
%
	\maketitle
	
	\begin{abstract}
	Let $\mbox{TFAG}$ be the theory of torsion-free abelian groups. We show that if there is no countable transitive model of $ZFC^- + \kappa(\omega)$ exists, then $\mbox{TFAG}$ is $a \Delta^1_2$-complete; in particular, this is consistent with $ZFC$. We define the $\alpha$-ary Schr\"{o}der- Bernstein property, and show that $\mbox{TFAG}$ fails the $\alpha$-ary Schr\"{o}der-Bernstein property for every $\alpha < \kappa(\omega)$. We leave open whether or not $\mbox{TFAG}$ can have the $\kappa(\omega)$-ary Schr\"{o}der-Bernstein property; if it did, then it would not be $a \Delta^1_2$-complete, and hence not Borel complete.
	\end{abstract}
	\section{Introduction}

In their seminal paper \cite{FS}, Friedman and Stanley introduced \emph{Borel complexity}, a measure of the complexity of the class of countable models of a sentence $\Phi \in \mathcal{L}_{\omega_1 \omega}$. Let $\mbox{Mod}(\Phi)$ be the set of all countable models of $\Phi$ with universe $\mathbb{N}$ (or any other fixed countable set). Then $\mbox{Mod}(\Phi)$ can be made into a standard Borel space in a natural way. 

\begin{definition}
Suppose $\Phi$, $\Psi$ are sentences of $\mathcal{L}_{\omega_1 \omega}$. Then say that $\Phi \leq_B \Psi$ ($\Phi$ is Borel reducible to $\Psi$) if there is a Borel-measurable function $f: \mbox{Mod}(\Phi) \to \mbox{Mod}(\Psi)$ satisfying the following: for all $M_1, M_2 \in \mbox{Mod}(\Phi)$, $M_1 \cong M_2$ if and only if $f(M_1) \cong f(M_2)$. 

Say that $\Phi \sim_B \Psi$ ($\Phi$ and $\Psi$ are Borel bi-reducible) if $\Phi \leq_B \Psi$ and $\Psi \leq_B \Phi$.
\end{definition}

 One way to think about the definition of $\leq_B$ is that $f$ induces an injection from $\mbox{Mod}(\Phi)/\cong$ to $\mbox{Mod}(\Psi)/\cong$; in other words, we are comparing the \emph{Borel cardinality} of $\mbox{Mod}(\Phi)/\cong$ and $\mbox{Mod}(\Psi)/\cong$.

In \cite{FS}, Friedman and Stanley showed that there is a maximal class of sentences under $\leq_B$, namely the \emph{Borel complete sentences}. For example, the theories of graphs, groups, rings, linear orders, and trees are all Borel complete. This provides a way to answer the question ``Is it possible to classify the countable models of $\Phi$" negatively in a precise sense: if $\Phi$ is Borel complete, then classifying the countable models of $\Phi$ is as hard as classifying arbitrary countable structures. 

In \cite{FS}, Friedman and Stanley leverage the Ulm analysis \cite{Ulm} to show that torsion abelian groups are far from Borel complete. They then pose the following question:

\vspace{2 mm}

\noindent \textbf{Question.} Let $\mbox{TFAG}$ be the theory of torsion-free abelian groups. Is $\mbox{TFAG}$ Borel complete?

\vspace{2 mm}

This has attracted considerable attention, but has nonetheless remained open. The following theorem of Hjorth \cite{Hjorth1} is the best known so far, where $(\Phi_\alpha: \alpha < \omega_1)$ is the Friedman-Stanley tower:

\begin{theorem}\label{HjorthFirst} 
$\Phi_\alpha \leq_B \mbox{TFAG}$ for every $\alpha < \omega_1$.
\end{theorem}

This means that if $\mbox{TFAG}$ is not Borel complete, then it represents a very new phenomenon. In fact, in \cite{FS}, Friedman and Stanley separately described the following question as one of the basic open problems of the general theory: if $\Phi$ is a sentence of $\mathcal{L}_{\omega_1 \omega}$ and if $\Phi_\alpha \leq_B \Phi$ for each $\alpha < \omega_1$, must $\Phi$ be Borel complete?

In Section~\ref{frames}, we give a uniform treatment of the main currently known techniques of coding information into abelian groups. The basic idea for these codings is old, dating at least to \cite{Hjorth1} and \cite{ShelahEkloff}; namely, we start with a free abelian group, and then tag various subgroups by making the elements infinitely divisible by particular primes. However, to make the coding more robust we adopt an idea of \cite{Goodrick}, replacing the use of primes by an algebraically independent sequence of $p$-adic integers for a fixed prime $p$. As a first application, we show the following, where $\mbox{AG}$ is the theory of abelian groups:

\begin{theorem}
$\mbox{TFAG} \sim_B \mbox{AG}$. Further, if $R$ is any countable ring, then $R\mbox{-mod}$, the theory of left $R$-modules, has $R\mbox{-mod} \leq_B \mbox{AG}$.
\end{theorem}

In Section~\ref{refinement}, we expand on Hjorth's proof of Theorem~\ref{HjorthFirst}. To state our results we need to introduce some more terminology.

\begin{definition}
By $ZFC^-$, we mean $ZFC$ without the power-set axiom, but where we strengthen replacement to collection and we strengthen choice to the well-ordering principle; this is as in \cite{ZFCminus}.

$\kappa(\omega)$ is the least cardinal $\kappa$ such that $\kappa \rightarrow (\omega)^{<\omega}_2$. This makes sense even in models of $ZFC^-$ (or less).

\end{definition}

\begin{definition}
Sometimes natural reductions that arise require transfinite recursion, and thus are not Borel. A coarser notion of reduciblity that allows for this is absolute $\Delta^1_2$ reducibility, denoted $a \Delta^1_2$. This notion has been studied, for instance, by Hjorth in Chapter 9 of \cite{Hjorth2}. Namely: suppose $\Phi, \Psi$ are sentences of $\mathcal{L}_{\omega_1 \omega}$. Then put $\Phi \leq_{a \Delta^1_2} \Psi$ if there is some function $f: \mbox{Mod}(\Phi) \to \mbox{Mod}(\Psi)$ with $\Delta^1_2$ graph, such that for all $M, N \in \mbox{Mod}(\Phi)$, $M \cong N$ if and only if $f(M) \cong f(N)$, and such that further, this continues to hold in any forcing extension. Explicitly, if $\sigma(x, y)$ is the $\Pi^1_2$ definition of the graph of $f$, and $\tau(x, y)$ is the $\Sigma^1_2$-definition of the graph of $f$, and if $\mathbb{V}[G]$ is a forcing extension, then $\sigma(x, y)$ and $\tau(x, y)$ coincide on $\mbox{Mod}(\Phi)^{\mathbb{V}[G]} \times \mbox{Mod}(\Psi)^{\mathbb{V}[G]}$ and define the graph of a function $f^{\mathbb{V}[G]}$, such that for all $M, N \in \mbox{Mod}(\Phi)^{\mathbb{V}[G]}$, $M \cong N$ if and only if $f^{\mathbb{V}[G]}(M) \cong f^{\mathbb{V}[G]}(N)$.

\end{definition}

%
%
%
%

	Using the basic idea of Theorem~\ref{HjorthFirst}, we are able to prove the following theorem in Section~\ref{refinement}:

	\begin{theorem}\label{EmbeddingWFFirst}
	Suppose there is no transitive model of $ZFC^- + \kappa(\omega) \mbox{ exists}$. Then $\mbox{Graphs} \leq_{a\Delta^1_2} \mbox{TFAG}$.
	\end{theorem}
	
	\begin{corollary}\label{ConsistentlyTFAGIsBadFirst}
	It is consistent with $ZFC$ that $\mbox{Graphs} \leq_{a \Delta^1_2} \mbox{TFAG}$, and hence that $\mbox{TFAG}$ is $a \Delta^1_2$-complete.
	\end{corollary}

It is natural to ask whether the set-theoretic hypothesis is necessary. For instance, the second author can show in \cite{BorelCompAndSB} that if $\kappa(\omega)$ exists, then a key part of the proof of Theorem~\ref{EmbeddingWFFirst} fails, namely, the conclusion of Theorem~\ref{DisparateTreesLemma}. This failure suggests the following question: are models of $\mbox{TFAG}$ controlled by some sort of biembeddability invariants? We investigate this question in Section~\ref{SBFailure}:

The Schr\"{o}der-Bernstein property is the simplest way that biembeddability can control isomorphism. This notion was originally introduced by Nurmagambetov \cite{Nurm1}, \cite{Nurm2}, who defined that a complete first order theory $T$ has the Schr\"{o}der-Bernstein property in the class of all models if for all $M, N \models T$, if $M$ and $N$ are elementarily biembeddable, then $M \cong N$.  Goodrick investigated this property further, including in his thesis \cite{Goodrick2} where he proves that if $T$ has the Schr\"{o}der-Bernstein property in the class of all models, then $T$ is classifiable of depth 1, i.e. $I(T, \aleph_\alpha) \leq |\alpha+\omega|^{2^{\aleph_0}}$ for all $\alpha$.

For our purposes, we want to tweak the definition in several ways. First of all, elementary embedding is somewhat awkward to deal with outside the context of complete first order theories.

\begin{definition}
Suppose $M, N$ are $\mathcal{L}$-structures. Then $f: M \leq N$ is an embedding if $f$ the following holds: whenever $R$ is a relation symbol of $\mathcal{L}$ then $f[R^M] \subseteq [R^N]$, and whenever $F$ is a function symbol of $\mathcal{L}$ then $f \circ F^M = F^N \circ f$. Say that $M \leq N$ if there is an embedding $f: M \to N$. Also, say that $(M, \overline{a}) \leq (N, \overline{b})$ if there is an embedding $f: M \leq N$ with $f(\overline{a}) = \overline{b}$. Finally, say that $M \sim N$ if $M \leq N \leq M$ and say that $(M, \overline{a}) \sim (N, \overline{b})$ if $(M, \overline{a}) \leq (N, \overline{b}) \leq (M, \overline{a})$.
\end{definition}

In the context of groups, we will only want to consider injective embeddings; formally then, we add a unary predicate for $\{(a, b): a \not= b\}$.

The following is what we mean by Schr\"{o}der-Bernstein property:

\begin{definition}
Suppose $\Phi$ is a sentence of $\mathcal{L}_{\omega_1 \omega}$. Then say that $\Phi$ has the Schr\"{o}der-Bernstein property if whenever $M, N$ are countable models of $\Phi$, if $M \sim N$ then $M \cong N$.
\end{definition}

This fails for $\mbox{TFAG}$, as first proved by Goodrick \cite{Goodrick2}. Recently, Calderoni and Thomas have shown in \cite{Thomas2} that the relation of biembeddability on models of $TFAG$ is $\Sigma^1_1$-complete, which is as bad as possible.

However, the proof of Theorem~\ref{EmbeddingWFFirst} suggests a weaker property: perhaps a group $G \models \mbox{TFAG}$ is determined by $\{(G, a)/\sim: a \in G\}$. We will call this the $1$-ary Schr\"{o}der Bernstein property. In Section~\ref{SBFailure}, we generalize this further to the $\alpha$-ary Schr\"{o}der-Bernstein property, for any ordinal $\alpha$; the $0$-ary Schr\"{o}der-Bernstein property is the Schr\"{o}der-Bernstein property.

The second author proves in \cite{BorelCompAndSB}:

\begin{theorem}
Suppose $\kappa(\omega)$ exists, and suppose $\alpha$ is an ordinal. If $\Phi$ is a sentence of $\mathcal{L}_{\omega_1 \omega}$ with the $\alpha$-ary Schr\"{o}der-Bernstein property, then $\Phi$ is not $a\Delta^1_2$-complete (and hence not Borel complete).
\end{theorem}

In Section~\ref{SBFailure}, we prove the following:

\begin{theorem}\label{SBFailureThm}
For every $\alpha < \kappa(\omega)$, $\mbox{TFAG}$ fails the $\alpha$-ary Schr\"{o}der-Bernstein property.
\end{theorem}

The construction breaks down at $\kappa(\omega)$, so the following remains open:

\vspace{2 mm}

\noindent \textbf{Question.} Does $\mbox{TFAG}$ have the $\kappa(\omega)$-ary Schr\"{o}der-Bernstein property?

\vspace{2 mm}

\noindent \textbf{Acknowledgements.} We would like to thank Julia Knight for pointing out a gap in a previous version.
\section{Some Bireducibilities with $\mbox{TFAG}$}\label{frames}
Notation: If $X$ is a set and $G$ is a group we let $\oplus_X G$ denote the group of functions from $X$ to $G$ with finite support; so we consider $\oplus_X G \leq G^X$.

For $p$ a prime, $\mathbb{Z}[\frac{1}{p}]$ is the subring of $\mathbb{Q}$ generated by $\frac{1}{p}$; and similarly for sets of primes. $\mathbb{Z}_{(p)}$ (read: $\mathbb{Z}$ localized at the ideal $(p)$) is $\mathbb{Z}[\frac{1}{q}: q \not= p]$. Let $\mathbb{Z}_p$ be the p-adic integers, i.e. the completion of $\mathbb{Z}_{(p)}$ under the p-adic metric. Let $\mathbb{Q}_p$ be the field completion of $\mathbb{Z}_p$.

Given $G \leq H$ groups, say that $G$ is a pure subgroup of $H$ is for every $n < \omega$, $nH \cap G = nG$. If $p$ is a prime, say that $G$ is a $p$-pure subgroup of $H$ if for every $n < \omega$, $p^n H \cap G = p^n G$. 

The following is a generalization of Hjorth's notion of ``eplag."

\begin{definition}
Suppose $\I$ and $\J$ are countable index sets. Then let $\mathcal{L}_{\I, \J}$ be the language extending the language of abelian groups, with a unary predicate symbol $G_i$ for each $i \in I$, and a unary function symbol $\phi_j$ for each $j \in J$ (we will allow $\phi_j$ to be a partial function).

Let $\Omega_{\I, \J}$ be the infinitary $\mathcal{L}_{\mathcal{F}}$-sentence such that $(G, +, G_i: i \in I, \phi_j: j \in J) \models \Omega_{\I, \J}$ if and only if the following all hold:

\begin{itemize}
\item $(G, +) \equiv_{\infty \omega} \oplus_{\omega} \mathbb{Z}$;
\item Each $G_i$ is a subgroup of $G$;
\item Each $\mbox{dom}(\phi_j)$ is either equal to all of $G$, or else to some $G_i$;
\item Each $\phi_j: \mbox{dom}(\phi_j) \to G$ is a homomorphism.
\end{itemize}

Let $\Omega_{\I, \J}^p$ assert additionally that each $G_i$ is a pure subgroup of $G$. 
\end{definition}

Some important examples: the countable models of $\Omega_{\{0\}, 0}$ are of the form $(G, H)$ where $G$ is free abelian of infinite rank (i.e., isomorphic to $\oplus_\omega \mathbb{Z}$) and $H$ is a subgroup of $G$. The countable models of $\Omega_{0, \{0\}}$ are of the form $(G, \phi)$ where $G$ is free abelian of infinite rank and $\phi: G \to G$ is a homomorphism. The countable models of $\Omega_{\omega, 0}$ are of the form $(G, G_n: n < \omega)$, where $G$ is free abelian of infinite rank and each $G_n$ is a subgroup of $G$.

We aim to prove the following. Let $\mbox{AG}$ denote the theory of abelian groups.

\begin{theorem}\label{Bireducibilities}
Suppose $\I, \J$ are countable index sets, not both empty. Then $\Omega_{\I, \J}^p \sim_B \Omega_{\I, \J} \sim_B \mbox{TFAG} \sim_B \mbox{AG}$.
\end{theorem}

The proof will be via many lemmas.

\begin{lemma}\label{BiredLemma1}
$\mbox{TFAG} \leq_B \Omega_{\{0\}, 0}^p$ and $\mbox{AG} \leq_B \Omega_{\{0\}, 0}$.
\end{lemma}
\begin{proof}

We describe the essential features of the construction, leaving it to the reader to check that it is Borel when formulated as an operation on Polish spaces. Suppose $G$ is an (infinite) countable abelian group. Define $\phi: \oplus_{G} \mathbb{Z} \to G$ to be the augmentation map, that is given $a \in \oplus_{G}\mathbb{Z}$, let $\phi(a) = \sum_{b \in G} a(b) b$ (this is really a finite sum).  Let $K$ be the kernel of $\phi$. Thus $G \mapsto (\oplus_G \mathbb{Z}, K)$ works, using $G \cong \sum_{G} \mathbb{Z}/K$. This shows $\mbox{AG} \leq_B \Omega_{\{0\}, 0}$; but note that if $G$ is torsion-free, then $K$ will be pure, so we also get $\mbox{TFAG} \leq_B \Omega_{\{0\}, 0}^p$.
\end{proof}
%
%

\begin{lemma}\label{BiredLemma2}
$\Omega_{\{0\}, 0} \leq_B \Omega_{0, \{0\}}$. Hence, whenever $\I, \J$ are not both empty, $\Omega_{\{0\}, 0} \leq_B \Omega_{\I, \J}$ and $\Omega_{\{0\}, 0}^p \leq_B \Omega_{\I, \J}^p$.
\end{lemma}
\begin{proof}
Suppose $(G, H) \models \Omega_{\{0\}, 0}$ is a given countable model; so $G$ is free abelian of infinite rank and $H$ is a subgroup of $G$. Write $G' = G \times H'$, where $H' \cong H$; note that $H'$ and hence $G'$ is free abelian, since subgroups of free abelian groups are free. Define $\phi: G' \to G'$ via $\phi \restriction_G = 0$ and $\phi \restriction_{H'}: H' \cong H$. Then $(G, H) \mapsto (G', \phi)$ works, using $G = \mbox{ker}(\phi)$ and $H = \mbox{im}(\phi)$.

The second claim follows trivially (note $\Omega_{0, \{0\}}^p = \Omega_{0, \{0\}}$).
\end{proof}

\begin{lemma}\label{BiredLemma3}
For any countable index sets $\I, \J$, $\Omega_{\I, \J} \leq_B \Omega_{\omega, 0}$ and $\Omega_{\I, \J}^p \leq_B \Omega_{\omega, 0}^p$.
\end{lemma}
\begin{proof}
Write $\I' = \I \cup \J \cup \{*_0, *_1\}$ (we suppose this is a disjoint union). We show that $\Omega_{\I, \J} \leq_B \Omega_{\I', 0}$ and $\Omega_{\I, \J}^p \leq_B \Omega_{\I', 0}^p$.

Suppose $(G, G_i: i \in \I, \phi_j: j \in \J) \models \Omega_{\I, \J}$. Define $G' = G \times G$; for each $i \in \I$, define $G'_{i}$ to be the copy of $G_i$ in the first factor of $G'$; for each $j \in \J$, define $G'_j$ to be the graph of $\phi_j$; define $G'_{*_0} = G \times 0$;  and finally let $G'_{*_1}$ be the graph of the identify function $id_G: G \to G$. Then $(G', G'_i: i \in \I') \models \Omega_{\I', 0}$ works. Also note that if each $G_i$ is pure, then so is each $G'_{i'}$; this is because the graph of a partial homomorphism is pure if and only if its domain is pure.
\end{proof}

\begin{lemma}\label{BiredLemma4}
$\Omega_{\omega, 0} \leq_B \Omega_{\omega, 0}^p$.
\end{lemma}
\begin{proof}
By the preceding lemma, it suffices to find index sets $\I, \J$ such that $\Omega_{\omega, 0} \leq_B \Omega_{\I, \J}^p$. Write $\I = \omega \cup \{*\}$, write $\J = \omega$.

Suppose $(G, G_n: n < \omega) \models \Omega_{\omega, 0}$. We define $G' = G \times \oplus_{n < \omega} (\oplus_{G_n} \mathbb{Z})$. For each $n < \omega$ let $G'_n = \oplus_{G_n} \mathbb{Z}$; let $G'_{*} = G$. Finally, define $\phi_n: G'_n \to G'$ to be the augmentation map $\oplus_{G_n} \mathbb{Z} \mapsto G_n$. Then clearly $(G', G'_i: i \in \I, \phi_j: j \in \J)$ works ($G = G'_{*}$ and each $G_n = \mbox{Im}(\phi_n)$).
\end{proof}

Note that to finish the proof of Theorem~\ref{Bireducibilities}, it suffices to show that $\Omega_{\omega, 0}^p \leq_B \mbox{TFAG}$. Indeed, we would then have that  for any countable index sets $\I, \J$ not both empty, $\mbox{TFAG} \leq_B \Omega_{\{0\}, 0}^p \leq_B \Omega_{\I, \J}^p \leq_B \Omega_{\omega, 0}^p \leq_B \mbox{TFAG}$, and thus these are all equivalent; and similarly, $\mbox{AG} \leq_B \Omega_{\I, \J} \leq_B \Omega_{\omega, 0} \leq_B \Omega_{\omega, 0}^p \leq_B \mbox{AG}$, and so these are also all equivalent. 

This remaining reduction is more involved than the others; the basic idea for it is due to Goodrick \cite{Goodrick}. To begin, we need the following lemma. The point is that if $G$ is a $p$-pure subgroup of $\oplus_\omega \mathbb{Z}_p$, then the isomorphism type of $(\mathbb{Z}_pG, G)$ depends only on the isomorphism type of $G$, where $\mathbb{Z}_p G$ is the $\mathbb{Z}_p$-submodule of $\oplus_\omega \mathbb{Z}_p$ generated by $G$.

\begin{lemma}\label{GroupTheoryLemma}
Suppose $G$ is a $p$-pure subgroup of $\oplus_\omega \mathbb{Z}_p$. Then there is a $\mathbb{Z}_p$-module isomorphism $\phi: (\mathbb{Z}_p \otimes G)/(p^\infty (\mathbb{Z}_p \otimes G)) \to \mathbb{Z}_p G$ which is the identity on $G$, where $\mathbb{Z}_p \otimes G$ is the tensor product (over $\mathbb{Z}$).
\end{lemma}
\begin{proof}
Define $\psi(\gamma, a) = \gamma a$, going from $\mathbb{Z}_p \times G$ to $\mathbb{Z}_p G$. $\psi$ is clearly a $\mathbb{Z}$-bilinear map, so it induces a group homomorphism $\phi_0: \mathbb{Z}_p \otimes G \to \mathbb{Z}_p G$. Clearly $\phi_0$ is $0$ on $p^\infty (\mathbb{Z}_p \otimes G)$ so this induces a map $\phi:  (\mathbb{Z}_p \otimes G)/(p^\infty (\mathbb{Z}_p \otimes G)) \to \mathbb{Z}_p G$. We check this works. Clearly $\phi$ is surjective and the identity on $G$, and preserves the $\mathbb{Z}_p$-action. So it suffices to check the kernel of $\phi_0$ is $p^\infty (\mathbb{Z}_p \otimes G)$.

Given $\gamma \in \mathbb{Z}_p$ and $n < \omega$, let $\gamma \restriction_n \in \{0, \ldots, p^n-1\}$ be the unique element with $\gamma - \gamma \restriction_n \in p^n \mathbb{Z}_p$ (recall that $\mathbb{Z}_p$ is the completion of $\mathbb{Z}$ in the $p$-adic metric; so choose $(k_m: m < \omega)$ a sequence from $\mathbb{Z}$ converging to $\gamma$ and note that $k_m \mbox{ mod }p^n$ must eventually be constant).

Suppose $\sum_{i < n} \gamma_i a_i = 0$; we want to show $\sum_{i < n} \gamma_i \otimes a_i \in p^\infty(\mathbb{Z}_p \otimes G)$.  Note that for each $m$, $\sum_{i < n} \gamma_i a_i \in p^m (\oplus_\omega \mathbb{Z}_p)$. Hence, for each $m$, $b_m := \sum_{i < n} \gamma_i \restriction_m a_i \in p^m G$, using that $G$ is $p$-pure. Note that in $\mathbb{Z}_p \otimes G$, $ \sum_{i < n} \gamma_i \restriction_m \otimes a_i = 1 \otimes b_m$, since we can move all the $\gamma_i \restriction_m$'s to the right-hand side; and $1 \otimes b_m \in p^m(\mathbb{Z}_p \otimes G)$. Also, $1 \otimes b_m - \sum_{i < n} \gamma_i \otimes a_i \in (p^m \mathbb{Z}_p) \otimes G$, as it is equal to $\sum_{i < n} (\gamma_i \restriction_m - \gamma_i) \otimes a_i$. Thus $\sum_{i < n} \gamma_i \otimes a_i \in p^m(\mathbb{Z}_p \otimes G)$ for all $m$, as desired.
\end{proof}

Finally:

 \begin{lemma}~\label{HjorthThm}
$\Omega_{\omega, 0}^p \leq_B \mbox{TFAG}$.
 \end{lemma}
 \begin{proof}
 Let $p$ be a prime.
 
  Let $(\gamma_n: 1 \leq n < \omega)$ be a sequence of algebraically-independent elements of $\mathbb{Z}_p$ over $\mathbb{Q}$, such that each $\gamma_n$ is a unit of $\mathbb{Z}_p$ (in particular is not divisible by $p$).  Write $\gamma_0 = 1$. Note then that $(\gamma_n: n < \omega)$ is linearly independent over $\mathbb{Q}$.

 Let $(\oplus_{\omega} \mathbb{Z}, G_n: n < \omega) \models \Omega^p_{\omega, 0}$; we can suppose $G_0 = G_1= \oplus_\omega \mathbb{Z}$. Let $G$ be the $p$-pure subgroup of $\oplus_\omega \mathbb{Z}_p$ generated by $\bigcup_{n < \omega} \gamma_n G_n$ (that is, close off under addition, inverses, and division by $p$ within $\oplus_\omega \mathbb{Z}_p$). We want to check that the map $\overline{G} \mapsto G$ works.
 
 First, suppose $(\oplus_\omega \mathbb{Z}, G_n: n < \omega) \cong (\oplus_\omega \mathbb{Z}, G'_n: n < \omega)$; we want to verify that the corresponding groups $G, G'$ are isomorphic. Let $\phi$ be the isomorphism. Then $\phi$ lifts canonically to an isomorphism $\phi^*: \oplus_\omega \mathbb{Z}_p \cong \oplus_\omega \mathbb{Z}_p$ (let $(e_i: i < \omega)$ be the standard basis of $\oplus_\omega \mathbb{Z}$, define $\phi^*(\sum_{i} \gamma_i e_i) = \sum_i \gamma_i \phi(e_i)$, where $(e_i: i < \omega)$ is the standard basis of $\oplus_\omega \mathbb{Z}$; more abstractly, $\phi^* = 1 \otimes \phi$ where we view $\oplus_\omega \mathbb{Z}_p = \mathbb{Z}_p \otimes \oplus_\omega \mathbb{Z}$). Then clearly $\phi^* \restriction_{G}$ is an isomorphism onto $G'$.
 
For the reverse it suffices, by Lemma~\ref{GroupTheoryLemma}, to show we can canonically recover each $G_n$ from $(\mathbb{Z}_pG, G)$.

Note that every $a \in G$ can be written as $\sum_{n < \omega} \gamma_n p^{k(n)} b_n$, when $k(n) \in \mathbb{Z}$, $b_n \in G_n$ with all but finitely many  $b_n = 0$, and $k(n) = 0$ whenever $b_n = 0$. (Not all such sums are in $G$; $G$ contains such sums which are additionally in $\oplus_\omega \mathbb{Z}_p$.) We call this a representation of $a$ if each $p \not| b_n$.  Then representations are unique: for suppose $\sum_{n < \omega} \gamma_n p^{k(n)} b_n = \sum_{n < \omega}  \gamma_n p^{k'(n)} b'_n$. Let $i \in \omega$; then we have $\sum_{n < \omega} \left( p^{k(n)} b_n(i) - p^{k'(n)} b'_n(i) \right) \gamma_n = 0$. By linear independence of $(\gamma_n: n < \omega)$ this implies each $p^{k(n)} b_n(i) = p^{k'(n)} b'_n(i)$. Since this holds for each $i$ we have each $p^{k(n)} b_n = p^{k'(n)} b'_n$. Then by divisibility assumptions we have that each $b_n = b'_n$ and so each $k(n) = k'(n)$.

 Suppose $f \in \mathbb{Z}_p G$ and let $1 \leq m < \omega$. It suffices to show that $a \in G_m$ if and only if $a \in G$ and $\gamma_m a \in G$: left to right follows from our assumption that $\gamma_0 = 1$. For right to left: let $\sum_{n < \omega} \gamma_n p^{k(n)} b_n$ be the representation of $a$, and let $\sum_{n < \omega} \gamma_n p^{k'(n)} b'_n$ be the representation of $\gamma_m a$. Let $i \in \omega$. Then $\sum_{n < \omega} \gamma_m \gamma_n p^{k(n)} b_n(i) = \sum_{n < \omega} \gamma_n p^{k'(n)} b'_n(i)$. Note that the only time $\gamma_m \gamma_n = \gamma_{k}$ is when $n = 0$, $k = m$. Thus by linear independence of $(\gamma_n: n < \omega)\,^\frown \,(\gamma_m \gamma_n: 1 \leq n < \omega)$ we have that $b_n = 0$ for all $n \not= 0$, and $b'_n = 0$ for all $n \not= m$. In particular, $a = p^k b$ for some $b \in G_m$. Since $\oplus_\omega \mathbb{Z}$ is $p$-pure in $\oplus_\omega \mathbb{Z}_p$ and since $G_m$ is $p$-pure in $\oplus_\omega \mathbb{Z}$, we have that $a \in G_m$.
 \end{proof}

\begin{remark}\label{RemarkOnStar} It is easy to add to the list in Theorem~\ref{Bireducibilities}. For instance, we can additionally insist that each $\phi_j$ is a pure embedding, i.e. preserves the divisibility relations.

A much stronger condition is the following: let $\Omega_{\I, \J}^*$ be $\Omega_{\I, \J}$ together with the second-order assertion saying, given $(G, G_i: i\in I, \phi_j: j \in J)$, that there is a basis $\mathcal{B}$ of $G$ (as a $\mathbb{Z}$-module) such that each $G_i$ is spanned by basis elements of $\mathcal{B}$ and each $\phi_j$ takes basis elements to basis elements. All of the known complexity of $\mbox{TFAG}$ is also present in $\Omega_{\omega, \{0\}}^*$; see the next section.

\end{remark}

Finally, we aim towards showing that whenever $R$ is a countable ring, then $R\mbox{-mod}$ (the theory of left $R$-modules) is Borel reducible to $\mbox{AG}$. This will not be used in the remainder of the paper.

\begin{definition}Suppose $\I, \J$ are countable index sets. Let $\Omega^{-}_{\I, \J}$ be the $\mathcal{L}_{\I, \J}$-theory such that $(G, +, G_i, \phi_j: i \in \I, j \in \J) \models \Omega^-_{\I, \J}$ if and only if:

\begin{itemize}
\item $(G, +)$ is an abelian group;
\item Each $G_i$ is a subgroup of $G$;
\item Each $\mbox{dom}(\phi_j)$ is either all of $G$ or else some $G_i$;
\item Each $\phi_j: \mbox{dom}(\phi_j) \to G$ is a homomorphism.
\end{itemize}
 So the only difference with $\Omega_{\I, \J}$ is that we are no longer requiring $G \equiv_{\infty \omega} \oplus_\omega \mathbb{Z}$.
\end{definition}

\begin{theorem}
For all countable index sets $\I, \J$, we have $\Omega^-_{\I, \J} \sim_B \mbox{AG}$.
\end{theorem}
\begin{proof}
Clearly $\mbox{AG} \leq_B \Omega^-_{\I, \J}$. (Given $G \models \mbox{AG}$, let each $G_i = G$ and let each $\phi_j$ be the identity of $G$.) Also, we have by exactly the same argument as before that each $\Omega^-_{\I, \J} \leq_B \Omega^-_{\omega, 0}$. So it suffices to show that $\Omega^-_{\omega, 0} \leq_B \Omega_{\omega \cup \{*\}, 0}$.

Given $(G, G_n: n< \omega) \models \Omega^-_{\omega, 0}$ (that is, $G$ is an abelian group and each $G_n$ is a subgroup of $G$), write $G' = \oplus_G \mathbb{Z}$; let $G'_*$ be the kernel of the augmentation map $G' \to G$; and let $G'_n = G'_* + \oplus_{G_n} \mathbb{Z}$. Then $(G', G'_n: n < \omega, G'_*)$ works, using $G \cong G'/G_*$ via an isomorphism that takes each $G_n$ to $G'_n/G_*$.
\end{proof}

\begin{corollary}
Suppose $R$ is a countable ring. Then $R\mbox{-mod} \leq_B \mbox{AG}$. 
\end{corollary}
\begin{proof}
An $R$-module $(M, +, \cdot_r: r \in R)$ can be viewed as a model of $\Omega^-_{0, R}$, and this gives a reduction $R\mbox{-mod} \leq_B \Omega^-_{0, R}$.
\end{proof}

\section{Embedding Graphs into $\mbox{TFAG}$}\label{refinement}
In this section, we prove Theorem~\ref{EmbeddingWFFirst}: if there is no transitive model of $ZFC^- + \kappa(\omega) \mbox{ exists}$, then $\mbox{Graphs} \leq_{a\Delta^1_2} \mbox{TFAG}$. To begin, we introduce some terminology for colored trees.

\begin{definition}
A colored tree is a structure,  $(T, \leq, 0, c)$ where $(T, \leq)$ is a tree (of height at most $\omega$) with root $0$, and $c: T \to \omega$. We view these as model-theoretic structures, formally we can replace $c$ with a sequence of unary predicates. Let $\mbox{CT}$ be the theory of colored trees.

As notation, when we say $\mathcal{T}, \mathcal{S}$, etc. is a colored tree, then we will have $\mathcal{T} = (T, <_{\mathcal{T}}, 0_{\mathcal{T}}, c_{\mathcal{T}})$, $\mathcal{S} = (S, <_{\mathcal{S}}, 0_{\mathcal{S}}, c_{\mathcal{S}})$, etc., unless stated otherwise.

Suppose $\mathcal{T}$ and $\mathcal{T}'$ are two colored trees. Then say that $f: \mathcal{T} \leq \mathcal{T}'$ is an embedding of trees if for all $f(0_{\mathcal{T}}) = 0_{\mathcal{S}}$, and for all $s,t \in T$, $s \leq_{\mathcal{T}} t$ if and only if $f(s) \leq_{\mathcal{T'}} f(t)$, and also $c_{\mathcal{T}'}(f(s)) = c_{\mathcal{T}}(s)$. (We do not require that $f$ be injective.) Say that $\mathcal{T}$ and $\mathcal{T}'$ are tree-biembeddable ($\mathcal{T} \sim \mathcal{T}'$) if $\mathcal{T} \leq \mathcal{T}'$ and $\mathcal{T}' \leq \mathcal{T}$. (These definitions agree with the definitions form the introduction). 

If $\mathcal{T}$  and $t \in T$ then $\mbox{ht}(t)$ denotes its height in $\mathcal{T}$ (if there is ambiguity we will write $\mbox{ht}_{\mathcal{T}}(t)$). Let $\mathcal{T}_{\geq t}$ denote the subtree of all elements of $T$ bigger than or equal to $t$, with the induced coloring.

\end{definition}

We will now split the proof of Theorem~\ref{EmbeddingWFFirst} into two main subtheorems.

\begin{theorem}\label{ZTreesLemma}
There is a Borel map $f: \mbox{Mod}(\mbox{CT}) \to \mbox{Mod}(\mbox{TFAG})$ such that for all $\mathcal{T}, \mathcal{T}' \models \mbox{CT}$, if $\mathcal{T} \cong \mathcal{T}'$ then $f(\mathcal{T}) \cong f(\mathcal{T}')$, and if $f(\mathcal{T}) \cong f(\mathcal{T}')$ then $\mathcal{T} \sim \mathcal{T}'$. (In fact, we will get that for every $t \in T$, there is $t' \in T'$ of the same height with $\mathcal{T}_{\geq t} \sim \mathcal{T}'_{\geq t'}$, and conversely.) 
\end{theorem}

\begin{theorem}\label{DisparateTreesLemma}Suppose there is no transitive model of $ZFC^- + \kappa(\omega) \mbox{ exists}$. Then there is  an absolutely $\Delta^1_2$-reduction $g:\mbox{Mod}(\mbox{Graphs}) \to \mbox{Mod}(\mbox{CT})$ such that whenever $G, G'  \in \mbox{Mod}(\mbox{Graphs})$, if $G \not \cong G'$ then then $f(G) \not \sim f(G')$.
\end{theorem}

We are essentially following Hjorth's proof of Theorem~\ref{HjorthFirst} in \cite{Hjorth1}, although Theorem~\ref{Bireducibilities} will make our life easier. The second author shows in \cite{BorelCompAndSB} that if $\kappa(\omega)$ exists, then the conclusion of Theorem~\ref{DisparateTreesLemma} fails.

Before proceeding, note that it suffices to establish Theorem~\ref{ZTreesLemma} and Theorem~\ref{DisparateTreesLemma}. Indeed, let $h = f \circ g: \mbox{Mod}(\mbox{Graphs}) \to \mbox{Mod}(\mbox{TFAG})$. Clearly $f \circ g$ has a $\Delta^1_2$ graph, and preserves isomorphism; we need to check this remains true in forcing extensions. Suppose $\mathbb{V}[G]$ is a forcing extension. By the definition of absolute $\Delta^1_2$-reduction, $g^{\mathbb{V}[G]}$ still makes sense, and is a reduction from $\mbox{Graphs}$ to $\mbox{CT}$. The remaining properties of $f, g$ are preserved by Shoenfield's absoluteness theorem.
\vspace{2 mm}

\noindent \emph{Proof of Theorem~\ref{ZTreesLemma}.}

 Suppose $\mathcal{T}= (T, <_T, c_T) \models \mbox{CT}$.  We define a model $\mathcal{T} \otimes \mathbb{Z}$ of $\Omega_{\omega \times \omega, \{0\}}$. ($f$ will be the function $\mathcal{T} \mapsto \mathcal{T} \otimes \mathbb{Z}$.)  Let the underlying group of $\mathcal{T} \otimes \mathbb{Z}$ be $\oplus_T \mathbb{Z}$;  define the group homomorphism $\pi_\mathcal{T}: \oplus_T \mathbb{Z} \to \oplus_T \mathbb{Z}$ by $\pi_\mathcal{T}(a)(t) = \sum_{s \in \mbox{\scriptsize{succ}}_\mathcal{T}(t)} a(s)$, where $\mbox{succ}_{\mathcal{T}}(t)$ denotes the set of all immediate successors of $s$ in $\mathcal{T}$. Viewing $T \subseteq \oplus_T \mathbb{Z}$, note that $\pi_{\mathcal{T}}(0_{\mathcal{T}}) = 0$, and for all $s \not= 0_{\mathcal{T}}$, $\pi_{\mathcal{T}}(s)$ is the immediate predecessor of $s$. For each $n, i< \omega$ write $G_{\mathcal{T}, n, i} = \oplus_{t} \mathbb{Z}$, where the sum is over all $t \in T$ of height $n$ and with $c_{\mathcal{T}}(t) = i$. Let $\mathcal{T} \otimes \mathbb{Z}$ be the structure $(\oplus_\mathcal{T} \mathbb{Z}, G_{\mathcal{T}, n, i}, \pi_\mathcal{T}: n, i < \omega)$.
 
Let $\mbox{CT} \otimes \mathbb{Z}$ be the $\Sigma^1_1$-sentence describing the closure under isomorphism of $\{\mathcal{T} \otimes \mathbb{Z}: \mathcal{T} \models \mbox{CT}\}$. 

Note that it is obvious that if $\mathcal{T}_1 \cong \mathcal{T}_2$ then $\mathcal{T}_1 \otimes \mathbb{Z} \cong \mathcal{T}_2 \otimes \mathbb{Z}$.

 Fix some countable $\mathcal{T} \models \mbox{CT}$. We perform some analysis on $\mathcal{T} \otimes \mathbb{Z}$; write $G= \oplus_T \mathbb{Z}$.

 For each $\overline{i} = (i_m: m < n+1) \in \omega^{n+1}$, let $G_{\mathcal{T}, \overline{i}}$ be the subgroup of all $a
 \in G$ such that for each $m \leq n$, $\pi^m(a) \in G_{\mathcal{T}, i_{n-m}}$. Also let $G_{\mathcal{T}, \emptyset} = 0$. Note that $\pi$ takes $G_{\mathcal{T}, \overline{i}}$ to $G_{\mathcal{T}, \overline{i} \restriction_n}$, also $G$ is the direct sum of the various $G_{\mathcal{T}, \overline{i}}$'s. Further, $G_{\mathcal{T}, \overline{i}}$ is spanned by $\{t \in T: \mbox{ht}(t) = n, \overline{c}_\mathcal{T}(t) = \overline{i}\}$, where $\overline{c}_\mathcal{T}(t) = (c_\mathcal{T}(t \restriction_0), c_\mathcal{T}(t \restriction_1), \ldots, c_\mathcal{T}(t))$. 

 For each $a \in G_{\mathcal{T}, \overline{i}}$ nonzero, let $T^*_a$ denote the set of all $b$ such that for some $\overline{i} \subseteq \overline{j}$, $b \in G_{\mathcal{T}, \overline{j}}$ and $\pi_\mathcal{T}^{\lg(\overline{j})-\lg(\overline{i})}(b) = a$. If we define $c^*_a(b) = \overline{j}(\lg(\overline{j})-1)$, and if we let $b \leq_f b'$ if and only if some $\pi^m(b') = b$, then $(T^*_a, \leq_a, c^*_a) = \mathcal{T}^*_a$ is a colored tree.

We need to characterize the colored trees $\mathcal{T}^*_a$ up to biembeddability. This will be done in terms of products of trees:

\begin{definition}
If $(\mathcal{S}_k: k < k_*)$ are colored trees, then by the product $\prod_{k < k_*} \mathcal{S}_k$, we mean the colored tree whose elements are all sequences $(s_k: k < k_*)$, where for some $n < \omega$, each $s_k$ has height $n$, and for some $(i_m: m \leq n) \in \omega^{n+1}$, we have for all $m \leq n$, $c_{\mathcal{S}_k}(s_k \restriction_m) = i_m$. Then we define the color of $(s_k: k < k_*)$ to be $i_n$. 
\end{definition}

Clearly, $\prod_{k < k_*} \mathcal{S}_k \leq \mathcal{S}_{k'}$ for each $k' < k_*$, via projection onto the $k'$-factor. In fact, $\mathcal{T} \leq \prod_{k < k_*} \mathcal{S}_k$ if and only if $\mathcal{T} \leq \mathcal{S}_k$ for each $k < k_*$. This is because if $\mathcal{T} \leq \prod_{k < k_*} \mathcal{S}_k$, then we can compose with the projection maps to get $\mathcal{T} \leq \mathcal{S}_k$ for each $k$; and if $f_k:\mathcal{T} \leq \mathcal{S}_k$ for each $k < k_*$, we can define $f: \mathcal{T} \leq \prod_{k < k_*} \mathcal{S}_k$ via $f(t) = (f_k(t): k < k_*)$.

 \vspace{1 mm}
 
 \noindent \textbf{Claim 1.} Suppose $a \in G_{\mathcal{T}, \overline{i}}$ is nonzero; enumerate $\mbox{supp}(a) = \{t_k: k < k_*\}$. (Here, we are viewing $a \in \oplus_T \mathbb{Z}$ as a function from $T$ to $\mathbb{Z}$ of finite support $\mbox{supp}(a)$.) Then $\mathcal{T}^*_a \sim \prod_{k < k_*} \mathcal{T}_{\geq t_k}$. 
 \begin{proof}First we will define an embedding $f: \mathcal{T}^*_a \leq \prod_{k < k_*} \mathcal{T}_{\geq t_k}$. We will define $f(b)$ inductively on the height of $b \in \mathcal{T}^*_a$; our inductive hypothesis will be that $f(b) = (t_k: k < k_*)$ is a sequence from $\mbox{supp}(b)$, and if we let $\overline{i}$ be such that $b \in G_{\mathcal{T}, \overline{i}}$, then each $\overline{c}_{\mathcal{T}}(t_k) = \overline{i}$. 
 
 So we are given $b$ and $f(b) = (t_k: k < k_*)$. Suppose $i < \omega$ and $c \in G_{\mathcal{T}, \overline{i} i}$ satisfies that $\pi_{\mathcal{T}}(c) = b$. Then $\pi_{\mathcal{T}}[\mbox{supp}(c)] \supseteq \mbox{supp}(b)$, so for each $k < k_*$ we can find $s_k \in \mbox{supp}(c)$ with $\pi_{\mathcal{T}}(s_k) = t_k$. Clearly then we can define $f(c) = (s_k: k < k_*)$, and continue.

 For the reverse embedding $\prod_{k < k_*} \mathcal{T}_{\geq t_k} \leq \mathcal{T}^*_a$, write $a= \sum_{k < k_*} \lambda_k t_k$, and send $(s_k: k < k_*) \in \prod_{k < k_*} \mathcal{T}_{\geq t_k}$ to $\sum_{k < k_*} \lambda_k s_k \in \mathcal{T}^*_a$. 
 \end{proof}

 Given an $\omega$-labeled tree $\mathcal{S}$, let $G_{\mathcal{T}, \overline{i}, \mathcal{S}}$ be the set of all $a \in G_{\mathcal{T}, \overline{i}}$ such that $\mathcal{S} \leq T^*_a$, along with $a = 0$. From the preceding claim it is clear that $G_{\mathcal{T}, \overline{i},\mathcal{S}}$ is a subgroup of $G_{\mathcal{T}, \overline{i}}$. Also, let $G_{\mathcal{T}, \overline{i}, >\mathcal{S}}$ be the sum of all $G_{\mathcal{T}, \overline{i}, \mathcal{S}'}$, for $\mathcal{S} <^{\mbox{\small{ct}}} \mathcal{S}'$. 

 Note that if $a \in G_{\mathcal{T}, \overline{i}}$, then always $a \in G_{\mathcal{T}, \overline{i}, \mathcal{T}^*_a}$, but sometimes also $a \in G_{\mathcal{T}, \overline{i}, >\mathcal{T}^*_a}$. Say that $a$ is good if this is not the case, i.e. $ a\in G_{\mathcal{T}, \overline{i}, \mathcal{T}^*_a} \backslash G_{\mathcal{T}, \overline{i}, > \mathcal{T}^*_a}$.
 
 \vspace{1 mm}
 
 \noindent \textbf{Claim 2.} Suppose $a \in G_{\mathcal{T}, \overline{i}}$. Then $a$ is good if and only if $a$ is nonzero, and there is some $t \in \mbox{supp}(a)$ such that $\mathcal{T}^*_{a} \sim \mathcal{T}_{\geq t}$.
 
 \begin{proof}
 Enumerate $\mbox{supp}(a) = \{t_k: k < k_*\}$, and write $a = \sum_{k < k_*} \lambda_k t_k$. Then by Claim 1, $\mathcal{T}^*_{a} \leq \prod_{k < k_*} \mathcal{T}_{\geq t_k}$, so $\mathcal{T}^*_{a} \leq \mathcal{T}_{\geq t_k}$ for each $k < k_*$.
 
 If $a$ is good, then we cannot have each $\mathcal{T}^*_{a} < \mathcal{T}_{\geq t_k}$, so some $\mathcal{T}_{\geq t_k} \sim \mathcal{T}^*_{a}$ as desired. For the converse, suppose $t \in \mbox{supp}(a)$ satisfies that $\mathcal{T}^*_{a} \sim \mathcal{T}_{\geq t}$ . Suppose we write $a= \sum_{i < i_*} b_i$. Then $t \in \mbox{supp}(b_i)$ for some $i < i_*$. By Claim 1, $\mathcal{T}^*_{b_i} \leq \mathcal{T}_{\geq t}$, and thus $\mathcal{T}^*_{b_i} \not > \mathcal{T}_{\geq t} \sim \mathcal{T}^*_{a}$.
 \end{proof}

 In particular, if $a \in G_{\mathcal{T}, \overline{i}}$ is good then $\mathcal{T}^*_a \sim \mathcal{T}_{\geq t}$ for some $t \in T$, and so we can recover $\{\mathcal{T}_{\geq t} / \sim: t \in T, \mbox{ht}(t) = n\}$ from the isomorphism class of $\mathcal{T} \otimes \mathbb{Z}$, for each $n$. This concludes the proof of Theorem~\ref{ZTreesLemma}.

\vspace{2 mm}

Before continuing on to the proof of Theorem~\ref{DisparateTreesLemma}, we need some set-theoretic observations. 

First, we note that various familiar facts about $\kappa(\omega)$ continue to hold when the ambient set theory is just $ZFC^-$ (less suffices as well).  Recall that a cardinal $\kappa$ (in a model of $ZFC$) is totally indescribable if for every $n$, for every sentence $\phi$ in the language of set theory with an extra relation symbol, and for every $R \subseteq \mathbb{V}_\kappa$ with $(\mathbb{V}_{\kappa + n}, \in, R) \models \phi$, there is an $\alpha < \kappa$ such that $(\mathbb{V}_{\alpha + n}, \in, R \cap \mathbb{V}_\alpha) \models \phi$. This is a large cardinal notion; it implies that $\kappa$ is weakly compact. In fact, weak compactness is equivalent to this condition when restricted to $n = 1$ (see Theorem 6.4 of Kanamori \cite{Kanamori}, due to Hanf and Scott).

\begin{lemma}Work in $ZFC^-$.

\begin{itemize}
\item[(A)] Suppose $\kappa \rightarrow (\omega)^{<\omega}_2$ and $N$ is a transitive model of $ZFC^-$ containing $\kappa$ (possible a proper class). Then $(\kappa \rightarrow (\omega)^{<\omega}_2)^N$.
\item[(B)] If $\mathbb{V} = \mathbb{L}$ (we really just need global choice), and if $\kappa(\omega)$ exists, then $\kappa(\omega)$ is inaccessible (i.e., $\kappa(\omega)$ is a regular cardinal, and for all $\alpha < \kappa(\omega)$, $\mathcal{P}(\alpha)$ exists and has cardinality less than $\kappa(\omega)$). Thus, $\mathbb{L}_{\kappa(\omega)} = \mathbb{V}_{\kappa(\omega)}$ is a set model of $ZFC$.
\item [(C)] If $\mathbb{V} = \mathbb{L}$ and if $\kappa(\omega)$ exists, then $\mathbb{V}_{\kappa(\omega)} \models ``$There exist totally indescribable cardinals."
\item[(D)] If $\mathbb{V} = \mathbb{L}$, then $\kappa(\omega)$ is the least cardinal $\kappa$ such that whenever $f: [\kappa]^{<\omega} \to 2$, there is an increasing sequence $(\alpha_n: n < \omega)$ from $\kappa$ such that for all $n$, $f(\alpha_0, \ldots, \alpha_{n-1}) = f(\alpha_1, \ldots, \alpha_n)$.
\item[(E)] If $\mathbb{V} = \mathbb{L}$, then $\kappa(\omega)$ is the least cardinal $\kappa$ such that there is no antichain $(\mathcal{T}_\alpha: \alpha < \kappa(\omega))$ of $\omega$-colored trees; by an antichain I mean that for all $\alpha < \beta < \kappa(\omega)$, $\mathcal{T}_\alpha \not \leq \mathcal{T}_\beta$ and $\mathcal{T}_\beta \not \leq \mathcal{T}_\alpha$. (If $\kappa(\omega)$ does not exist then we just mean that for every cardinal $\kappa$, there is an antichain of length $\kappa$.)
\end{itemize}

\end{lemma}

Note that Corollary~\ref{ConsistentlyTFAGIsBadFirst} follows from Theorem~\ref{EmbeddingWFFirst} and (B). (C) provides a strengthening: it is consistent with $ZFC + \mbox{``There is a totally indescribable cardinal"}$ that $\mbox{Graphs} \leq_{a\Delta^1_2} \mbox{TFAG}$.
\begin{proof}
All of these are routine modifications of the case where the ambient set theory is $ZFC$. In the context of $ZFC$: (A) and (D) are due to Silver \cite{Silver1}. (B) is also due to Silver \cite{Silver2}, or see Corollary 7.6 of Kanamori \cite{Kanamori}. (C) is due to Silver and Reinhardt, see Exercise 9.18 of \cite{Kanamori}. (E) is due to Shelah \cite{ShelahTrees}; we provide a sketch of the proof.

First suppose $\kappa < \kappa(\omega)$. Choose some $f: [\kappa]^{<\omega} \to 2$ failing (D). For each $\alpha < \kappa$, we define a colored tree $\mathcal{T}_\alpha$ as follows. Namely, let $T_\alpha$ be all finite increasing sequences of ordinals from $\kappa$ whose first term is $\alpha$; let $<_{\mathcal{T}_\alpha}$ be initial segment. Let $c_{\mathcal{T}_\alpha}(s) = f(s)$. Let $\mathcal{S}_\alpha$ be $\mathcal{T}_\alpha$ together with the tree of descending sequences from $\alpha$, with the new elements all colored $2$.

Note that for all $\alpha_0 < \alpha_1 < \kappa$, $\mathcal{T}_{\alpha_0} \not \leq \mathcal{T}_{\alpha_1}$, as given an embedding $\rho: \mathcal{T}_\alpha \leq \mathcal{T}_\beta$, we can inductively find $\alpha_n: n < \omega$ such that for all $n$, $\rho(\alpha_i: i <n) = (\alpha_i: 1 \leq i \leq n+1)$; but this clearly contradicts the hypothesized property of $f$. From this it follows that $(\mathcal{S}_\alpha: \alpha < \kappa)$ is the desired antichain.

In the other direction, suppose $(\mathcal{T}_\alpha: \alpha < \kappa(\omega))$ is a sequence of colored trees. Write $\kappa = \kappa(\omega)$; choose an elementary substructure $H \leq (\mathbb{V}_\kappa, \ldots)$ (using $<_{\mathbb{L}}$) such that $H$ is the Skolem hull of an infinite set of indiscernible ordinals $\{\alpha_n: n < \omega\}$. Then it is easy to check that $\mathcal{T}_{\alpha_0} \leq \mathcal{T}_{\alpha_1}$.
\end{proof}

We can now finish.

\vspace{2 mm}

\noindent \emph{Proof of Theorem~\ref{DisparateTreesLemma}.}

Suppose $A$ is a hereditarily countable set. We describe a colored tree $\mathcal{T}_{A} = (T_A, <_A, c_A)$, and then show that for all $A \not= A'$ then $\mathcal{T}_A \not \sim \mathcal{T}_{A'}$. Moreover, the operation $A \mapsto \mathcal{T}_A$ will be absolute to transitive models of $ZFC^-$. 

Before proceeding, we indicate how we finish. Given a graph $G \in \mbox{Mod}(\mbox{Graphs})$, let $g(G)$ be the $<_{\mathbb{L}[G]}$-least element of $\mbox{Mod}(\mbox{CT})$ which is isomorphic to $\mathcal{T}_{\mbox{css}(G)}$, where $\mbox{css}(G)$ is the canonical Scott sentence of $G$. (Note that $\mathcal{T}_{\mbox{\css}(G)} \in (\HC)^{\mathbb{L}[G]}$ since $(\HC)^{\mathbb{L}[G]} \models ZFC^-$, so $\mathcal{T}_{\mbox{css}(G)}$ does have models with universe $\omega$ in $\mathbb{L}[G]$.) Clearly, for any $G, G' \models \mbox{CT}$, if $G \cong G'$ then $\mbox{css}(G) = \mbox{css}(G')$ so $g(G) = g(G')$, and if $G \not \cong G'$ then $\mbox{css}(G) \not= \mbox{css}(G')$ and so $g(G) \not \sim g(G')$. To finish, note that $g$ is an absolutely $\Delta^1_2$-reduction, since it is computed correctly in any countable transitive model of $ZFC^-$.

So we define $A \mapsto \mathcal{T}_A$. Let $A$ be given, and let $\alpha = \mbox{rnk}(A)$, where $\mbox{rnk}$ is foundation rank. Let $(\mathcal{S}_\beta: \beta \leq \alpha)$ be the $<_{\mathbb{L}}$-least antichain of colored trees indexed by $\alpha+1$. This is computed correctly in any transitive model of $ZFC^-$, since if $M$ is any transitive model of $ZFC^-$ with $\alpha \in M$, then $\mathbb{L}^M$ does not believe that $\kappa(\omega)$ exists, and so $\mathbb{L}^M$ can find a $<_{\mathbb{L}^M}$-least sequence $(\mathcal{S}_\beta: \beta \leq \alpha)$ such that $\mathbb{L}^M \models (\mathcal{S}_\beta: \beta \leq \alpha)$ is an antichain. But the property of being an antichain of colored trees of length $\alpha+1$ is absolute to models of $ZFC^-$; thus $(\mathcal{S}_\beta: \beta \leq \alpha)$ is the $<_{\mathbb{L}}$-least antichain of colored trees indexed by $\alpha+1$.

We define a preliminary colored tree $\mathcal{T}_{0, A} = (T_{0, A}, <_{0, A}, c_{0, A})$. Let $(T_{0, A}, <_{0, A})$ be the tree of all nonempty finite sequences $(a_0, \ldots, a_{n})$ from $\mbox{tcl}(A \cup \{A\})$ such that $a_0 = A$ and $\mbox{rnk}(a_0) > \mbox{rnk}(a_1) > \ldots > \mbox{rnk}(a_n)$. Given $(a_0, \ldots, a_n) \in T_{0, A}$, let $c_{0, A}(a_0, \ldots, a_n) = 0$ if $a_{n-1} \in a_n$, and $c_{0, A}(a_0, \ldots, a_n) = 1$ otherwise. Let $\mathcal{T}_A$ be obtained from $\mathcal{T}_{0, A}$ as follows: above each $(a_0, \ldots, a_n) \in T_{0, A}$, put a copy of $(S_\beta, <_{S_\beta})$, where $\beta$ is the foundation rank of $a_n$; given $t \in S_\beta$, let the color of the copy of $t$ above $(a_0, \ldots, a_n)$ be $c_{\mathcal{S}_\beta}(t) + 2$.

Suppose $\mathcal{T}_A \sim \mathcal{T}_{A'}$. Let $\alpha = \mbox{rnk}(A)$ and let $\alpha' = \mbox{rnk}(A')$. Let $f: \mathcal{T}_A \leq \mathcal{T}_{A'}$ and $f': \mathcal{T}_{A'} \leq \mathcal{T}_A$ witness that $\mathcal{T}_A \sim \mathcal{T}_{A'}$. Note that $f \restriction_{T_{A, 0}}$ and $f' \restriction_{T_{A', 0}}$ witness that $\mathcal{T}_{A, 0}$ and $\mathcal{T}_{A', 0}$ are biembeddable; since $\mathcal{T}_{A, 0}$ is well-founded of rank $\alpha$, and $\mathcal{T}_{A', 0}$ is well-founded of rank $\alpha'$, this implies $\alpha = \alpha'$. Let $(\mathcal{S}_\beta: \beta \leq \alpha)$ be as above.

Now, consider the embedding $h := f' \circ f: \mathcal{T}_A \leq \mathcal{T}_A$. I claim that $h \restriction_{\mathcal{T}_{0, A}}$ must be the identity. This suffices, since it implies $\mathcal{T}_{0, A} \cong \mathcal{T}_{0, A'}$ and hence $A = A'$.

Suppose $(a_0, \ldots, a_n) \in T_{0, A}$; write $\beta = \mbox{rnk}(a_n)$ and write $h(a_0, \ldots, a_n)= (b_0, \ldots, b_n)$. We show by induction on $\beta$ that $a_n = b_n$; this suffices. Note that $\mathcal{S}_\beta \leq \mathcal{S}_{\mbox{rnk}(b_n)}$, and hence $\mbox{rnk}(b_n) = \beta$ also (this is the key point!).

If $\beta = 0$, then $a_n = b_n = \emptyset$. Suppose we have verified the claim for all $\gamma < \beta$. We show that for every $a \in \mbox{tcl}(A \cup \{A\})$ with $\mbox{rnk}(a) < \beta$, we have that $a \in a_n$ if and only if $a \in b_n$. Indeed, suppose $a$ is given. Write $h(a_0, \ldots, a_n, a) = (b_0, \ldots, b_n, b)$. By construction of the coloring, we have that $a \in a_n$ if and only if $b \in b_n$; but by the inductive hypothesis, we have that $a = b$.

\section{Schr\"{o}der-Bernstein Properties for $\mbox{TFAG}$}\label{SBFailure}

We repeat a bit from the introduction.

\begin{definition}
Suppose $M, N$ are $\mathcal{L}$-structures. Then $f: M \leq N$ is an embedding if $f$: whenever $R$ is a relation symbol of $\mathcal{L}$ then $f[R^M] \subseteq [R^N]$, and whenever $F$ is a function symbol of $\mathcal{L}$ then $f \circ F^M = F^N \circ f$. Say that $M \leq N$ if there is an embedding $f: M \to N$. Also, say that $(M, \overline{a}) \leq (N, \overline{b})$ if there is an embedding $f: M \leq N$ with $f(\overline{a}) = \overline{b}$. Finally, say that $M \sim N$ if $M \leq N \leq M$ and say that $(M, \overline{a}) \sim (N, \overline{b})$ if $(M, \overline{a}) \leq (N, \overline{b}) \leq (M, \overline{a})$.
\end{definition}
In the context of groups, we will only want to consider injective embeddings; formally then, we add a unary predicate for $\{(a, b): a \not= b\}$.

\begin{definition}
Suppose $\Phi$ is a sentence of $\mathcal{L}_{\omega_1 \omega}$. Then say that $\Phi$ has the Schr\"{o}der-Bernstein property if whenever $M, N$ are countable models of $\Phi$, if $M \sim N$ then $M \cong N$.
\end{definition}

This fails for $\mbox{TFAG}$, as first proved by Goodrick \cite{Goodrick2} and in a strong form by Calderoni and Thomas \cite{Thomas2}. Nonetheless, the statement of Theorem~\ref{ZTreesLemma} suggests a weaker property: is a group $G \models \mbox{TFAG}$ is determined by $\{(G, a)/\sim: a\in G\}$? We will call this the $1$-ary Schr\"{o}der Bernstein property. Generalizing further:

\begin{definition}
Suppose $M, N$ are $\mathcal{L}$-structures, and $\overline{a} \in M$, $\overline{b} \in M$ are tuples of the same length. By induction on the ordinals we define what it means for $(M, \overline{a}) \sim_\alpha (N, \overline{b})$.

\begin{itemize}
\item $(M, \overline{a}) \sim_0 (N, \overline{b})$ if and only if $(M, \overline{a}) \sim (N, \overline{b})$.
\item For $\delta$ limit, $(M, \overline{a}) \sim_{\delta} (N, \overline{b})$ if and only if $(M, \overline{a}) \sim_\alpha (N, \overline{b})$ for all $\alpha < \delta$.
\item $(M, \overline{a}) \sim_{\alpha+1} (N, \overline{b})$ if and only if for all $a \in M$ there is $b \in M$ with $(M, \overline{a} a) \sim_\alpha (N, \overline{b} b)$, and conversely.
\end{itemize}

Say that $M \sim_\alpha N$ if $(M, \emptyset) \sim_\alpha (N, \emptyset)$.
\end{definition}

Note the similarity between these clauses and the clauses for defining $\equiv_{\alpha \omega}$; the only change is to the base 
case.

\begin{definition}
Suppose $\alpha < \omega_1$. Then say that $\Phi$ has the $\alpha$-ary Schr\"{o}der-Bernstein property if for all countable models $M, N \models \Phi$, if $M \sim_\alpha N$ then $M \cong N$.
\end{definition}

The notion of $\alpha$-ary Schr\"{o}der-Bernstein property can be extended to $\alpha \geq \omega_1$, with some care:

\begin{definition}
Suppose $\Phi$ is a sentence of $\mathcal{L}_{\omega_1 \omega}$. A pinned name for a model of $\Phi$ is a pair $(P, \dot{M})$, where $P$ is a forcing notion, $P \Vdash \dot{M} \in \mbox{Mod}(\check{\Phi})$, and $P \times P \Vdash \dot{M}_0 \cong \dot{M}_1$, where $\dot{M}_0$ is the copy of $\dot{M}$ in the first factor of $P \times P$, and $\dot{M}_1$ is the copy of $\dot{M}$ in the second factor of $P \times P$.

Suppose $(P, \dot{M})$ and $(Q, \dot{N})$ are pinned names for models $\Phi$, and $\alpha$ is an ordinal. Then say that $(P, \dot{M}) \sim_\alpha (Q, \dot{N})$ if $P \times Q \times R \Vdash \dot{M} \sim_\alpha \dot{N}$, where $R$ is some or any forcing notion which makes $\alpha, P, Q, \dot{M}, \dot{N}$ all countable. Say that $(P, \dot{M}) \cong (Q, \dot{N})$ if $P \times Q \Vdash \dot{M} \cong \dot{N}$.

Say that $\Phi$ has the $\alpha$-ary Schr\"{o}der-Bernstein property if for all pinned names $(P, \dot{M})$, $(Q, \dot{N})$ for models of $\Phi$, if $(P, \dot{M}) \sim_\alpha (Q, \dot{N})$ then $(P, \dot{M}) \cong (Q, \dot{N})$.
\end{definition}

This does not conflict with the previous definition, by a downward Lowenheim-Skolem argument; see \cite{BorelCompAndSB}. (In \cite{BorelCompAndSB}, canonical Scott sentences are used in place of pins, but this is equivalent.)

The following will serve as the only interface we need with the machinery of pins:

\begin{lemma}\label{SuffCondForFailSB}
Suppose $\Phi$ is a sentence of $\mathcal{L}_{\omega_1 \omega}$, and $\alpha$ is an ordinal. Suppose there are $M, N \models \Phi$ such that $M \sim_\alpha N$ but $M \not \equiv_{\infty \omega} N$. Then $\Phi$ fails the $\alpha$-ary Schr\"{o}der-Bernstein property.
\end{lemma}
\begin{proof}
Let $P_M$ be the set of all finite partial functions from $\omega$ to $M$, and let $\dot{f}_M$ be the $P_M$-name for the generic surjection from $\omega$ onto $\check{M}$ added by $P_M$. Let $P_N, \dot{f}_N$ be defined similarly. Then $(P_M, \dot{f}_M^{-1}(\check{M}))$ and $(P_N, \dot{f}_N^{-1}(\check{N}))$ are pinned names for models of $\Phi$, and it is easy to check that $(P_M, \dot{f}_M^{-1}(\check{M})) \sim_\alpha (P_N, \dot{f}_N^{-1}(\check{N}))$ but $(P_M, \dot{f}_M^{-1}(\check{M})) \not \cong (P_N, \dot{f}_N^{-1}(\check{N}))$.
\end{proof}

Looking at the statement of Theorem~\ref{ZTreesLemma}, it is reasonable to ask if $\mbox{TFAG}$ has the $1$-ary Schr\"{o}der-Bernstein property. This would have consequences for the complexity of $\mbox{TFAG}$, as the following theorem of the second author \cite{BorelCompAndSB} shows:

\begin{theorem}\label{SBFailureTheorem}
Suppose $\kappa(\omega)$ exists, and suppose $\alpha$ is an ordinal. If $\Phi$ is a sentence of $\mathcal{L}_{\omega_1 \omega}$ with the $\alpha$-ary Schr\"{o}der-Bernstein property, then $\Phi$ is not $a\Delta^1_2$-complete (and hence not Borel complete).
\end{theorem}

In this section, we prove Theorem~\ref{SBFailure}, namely: for every $\alpha < \kappa(\omega)$, $\mbox{TFAG}$ fails the $\alpha$-ary Schr\"{o}der-Bernstein property. The construction breaks down at $\kappa(\omega)$, so the following remains open:

\vspace{2 mm}

\noindent \textbf{Question.} Does $\mbox{TFAG}$ have the $\kappa(\omega)$-ary Schr\"{o}der-Bernstein property?

\vspace{2 mm}

In the remainder of this section, we prove the following:

\begin{theorem}\label{SBFailureTheorem}
Suppose $\kappa(\omega)$ does not exist. Then for every ordinal $\alpha$, $\mbox{TFAG}$ fails the $\alpha$-ary Schr\"{o}der-Bernstein property.
\end{theorem}
Note that Theorem~\ref{SBFailureThm} follows: for every $\alpha < \kappa(\omega)$, $\mbox{TFAG}$ fails the $\alpha$-ary Schr\"{o}der-Bernstein property. This is because we can always apply Theorem~\ref{SBFailureTheorem} in $\mathbb{V}_{\kappa(\omega)} = H(\kappa(\omega))$.

So, in the remainder of this section, suppose $\kappa(\omega)$ does not exist; equivalently, for every cardinal $\lambda$, there is an antichain of colored trees of length $\alpha$. 

First of all, we note the following lemma:

\begin{lemma}\label{RedToFrames}
Suppose $\I, \J$ are countable index sets, not both empty; let $F: \Omega_{\I, \J}^p \leq_B \mbox{TFAG}$ be the Borel reduction from the proof of Theorem~\ref{Bireducibilities} (that is, the composition of the reductions from Lemma~\ref{BiredLemma3} and Lemma~\ref{HjorthThm}). Suppose $\overline{G}^0, \overline{G}^1 \in \mbox{Mod}(\Omega_{\I, \J})$ and $\alpha < \omega_1$. If $\overline{G}^0 \sim_{2 \cdot (\omega \cdot \alpha)} \overline{G}^1$, then $F(\overline{G}^0) \sim_{\alpha} F(\overline{G}^1)$. 

Hence, if $\Omega^p_{\I, \J}$ fails the $\alpha$-ary Schr\"{o}der-Bernstein property for every ordinal $\alpha < \kappa(\omega)$, then so does $\mbox{TFAG}$.
\end{lemma}
\begin{proof}
	The final claim follows, since the first part continues to hold in forcing extensions.

Write $\I' = \I \cup \J \cup \{*_0, *_1\}$ (we suppose this is a disjoint union).

Let $F_0: \Omega_{\I, \J}^p \leq_B \Omega_{\I', 0}^p$ be as in Lemma~\ref{BiredLemma3} and let $F_1: \Omega_{\omega, 0}^p \leq_B \mbox{TFAG}$ be as in Lemma~\ref{HjorthThm}. 

First we look at $F_0$. We recap the definition of $F_0$, for the reader's convenience. 
Suppose $\overline{G} = (G, G_i: i \in \I, \phi_j: j \in \J) \models \Omega_{\I, \J}^p$ is countable. Define $G' = G \times G$; for each $i \in \I$, define $G'_{i}$ to be the copy of $G_i$ in the first factor of $G'$; for each $j \in \J$, define $G'_j$ to be the graph of $\phi_j$; define $G'_{*_0} = G \times 0$;  and finally let $G'_{*_1}$ be the graph of the identify function $id_G: G \to G$. Then $F(G, G_i: i \in I, \phi_j: j \in J)$ is $\overline{G}' = (G', G'_{i'}: i' \in \I')$ (suppressing the coding that arranges everything to have universe $\omega$).

Suppose $\overline{G}_0, \overline{G}_1 \models \Omega_{\I, \J}^p$ are countable, and define $\overline{G}'_0, \overline{G}'_1$ as above. Then it is easy to check that for all $((a_i^{0}, a_i^1): i < i_*)$ from $\overline{G}'_0$ and all $(b_i^0, b_i^1): i < i_*)$ from $\overline{G}'_1$, if $f: (\overline{G}_0, (a_i^j: i < i_*, j < 2)) \leq (\overline{G}_1, (b_i^j: i< i_*, j < 2)$, then $f \times f: (\overline{G}'_0, ((a_i^0, a_i^1): i < i_*)) \leq (\overline{G}'_1, ((b_i^0, b_i^1): i < i_*))$. From this it follows by an easy inductive argument that for all $\beta < \omega_1$, if $(\overline{G}_0, (a_i^j: i < i_*, j < 2)) \sim_{2 \cdot \beta} (\overline{G}_1, (b_i^j: i< i_*, j < 2)$, then $(\overline{G}'_0, ((a_i^0, a_i^1): i < i_*)) \sim_\beta (\overline{G}'_1, ((b_i^0, b_i^1): i < i_*))$. 

%

Next we look for $F_1$. Let $(\gamma_n: 1 \leq n < \omega)$ be as in Lemma~\ref{HjorthThm}, i.e. a sequence of algebraically independent units of $\mathbb{Q}_p$; and let $\gamma_0 =1$. Let $\overline{G} = (\oplus_\omega \mathbb{Z}, G_n: n < \omega)$ be a countable model of $\Omega_{\omega, 0}^p$; we only consider the case where $G_0 = G_1 = \oplus_\omega \mathbb{Z}$, without loss of generality. Then recall $F_1(\overline{G})$ is (isomorphic to) $G$, where $G$ is the $p$-pure subgroup of $\oplus_\omega \mathbb{Z}_p$ generated by $\bigcup_n \gamma_n G_n$. Recall that every $a \in G$ can be written as a sum $a = \sum_{n < \omega} \gamma_n p^{k(n)} b_n$, where each $k(n) \in \mathbb{Z}$, $b_n \in G_n$ and all but finitely many $k(n), b_n$ are $0$. Say that this is a weak representation of $a$ (it may not be a full representation; we don't require that $p \not | b_n$ in $G_n$.)

Suppose $\overline{G}^j = (\oplus_\omega \mathbb{Z}, G^j_n: n < \omega)$ are countable models of $\Omega_{\omega, 0}^p$ for $j < 2$; let $G^0, G^1$ be defined from $\overline{G}^0, \overline{G}^1$ as above. Suppose $f: \overline{G}^0 \leq \overline{G}^1$. Define $f_*: \oplus_\omega \mathbb{Z}_p \to \oplus_\omega \mathbb{Z}_p$ via $f_*(\sum_n \gamma_n e_n) = \sum_n \gamma_n f(e_n)$, where $(e_n: n < \omega)$ is the standard basis. Moreover, $f_* \restriction_{G^0}: G^0 \leq G^1$, since $f_*$ preserves the action of $\mathbb{Z}_p$. 

Suppose $(a_i: i < i_*)$ is a sequence from $\oplus_{\omega} \mathbb{Z}$, and suppose $(a'_i: i < i_*)$ is a sequence from $\oplus_\omega \mathbb{Z}$. Suppose for each $i < i_*$, $a_i = \sum_{n \in \Gamma_i} \gamma_n p^{k_i(n)} b_{i,n}$ is a weak representation with respect to $\overline{G}^0$, and $a'_i = \sum_{n \in \Gamma_i} \gamma_n p^{k_i(n)} b'_{i,n}$ is a weak representation with respect to $\overline{G}^1$, for finite sets $\Gamma_i\subset \omega$. Suppose finally that $f: (\overline{G}^0, (b_{i, n}: n \in \Gamma_i, i < i_*)) \leq (\overline{G}^1, (b'_{i, n}: n \in \Gamma_i, i < i_*))$. Then note that each $f_*(p^{k_i(n)} b_{i, n}) =p^{k_i(n)} b'_{i, n}$, hence each $f_*(a_i) =a'_i$, hence $f_*: (G^0, (a_i: i < i_*)) \leq (G^1, (a'_i: i < i_*))$.

From this, an easy inductive argument shows that if $(\overline{G}^0, (b_{i, n}: n \in \Gamma_i, i < i_*)) \sim_{\omega \cdot \alpha} (\overline{G}^1, (b'_{i, n}: n \in \Gamma_i, i < i_*))$, then $(G^0, (a_i: i < i_*)) \sim_\alpha (G^1, (a'_i: i < i_*))$.

\end{proof}
Thus it suffices to show that some $\Omega_{\I, \J}^p$ fails the $\alpha$-ary Schr\"{o}der-Bernstein property for all $\alpha$.

For the next lemma, we make the obvious definitions for $\Omega_{\I, \J}^p$ in the case where the index sets are possibly uncountable. 

\begin{lemma}\label{CodeSeq}
Suppose $\kappa(\omega)$ does not exist. Suppose $\I, \J$ are index sets, and suppose $\overline{G}^0, \overline{G}^1 \models \Omega_{\I, \J}$. Then we can find $\mathbf{F}(\overline{G}^0), \mathbf{F}(\overline{G}^1) \models \Omega_{\omega \times \omega \cup \{0, 1\}, \{0, 1\}}^p$, such that $\overline{G}^0 \equiv_{\infty \omega} \overline{G}^1$ if and only if $\mathbf{F}(\overline{G}^0) \equiv_{\infty \omega} \mathbf{F}(\overline{G}^1)$, and for every ordinal $\beta$, if $\overline{G}^0 \sim_\beta \overline{G}^1$ then $\mathbf{F}(\overline{G}^0) \sim_\beta \mathbf{F}(\overline{G}^1)$.
\end{lemma}
\begin{proof}
We can suppose $\J = \emptyset$, by applying the construction from Lemma~\ref{BiredLemma3}. 

Choose $\lambda$ large enough so that $\I, \overline{G}^0, \overline{G}^1$ all are of size at most $\lambda$. We can suppose $\I = \lambda$. 

Let $(\mathcal{T}_\gamma: \gamma < \lambda)$ be a family of pairwise-non-biembeddable colored trees. Let $\mathcal{T}$ be the colored tree such that $c_\mathcal{T}(0) =0$ (say), and for each $\gamma< \lambda$, there are $\lambda$-many $t \in T$ of height $1$ such that $\mathcal{T}_{\geq t} \cong \mathcal{T}_{\gamma}$, and for each $t \in T$ of height $1$, $\mathcal{T}_{\geq t}$ is isomorphic to some such $\mathcal{T}_\gamma$.

Recall the definition of $\mathcal{T} \otimes \mathbb{Z} = (G_{\mathcal{T}}, G_{\mathcal{T}, n, i}, \pi: n, i < \omega) \models \Omega_{\omega \times \omega, \{0\}}^p$ from Theorem~\ref{ZTreesLemma}. For each $\gamma < \lambda$, let $\mathcal{E}_\gamma$ be the set of all $t \in T$ of height $1$ such that $\mathcal{T}_{\geq t} \cong \mathcal{T}_\gamma$. Let $\hat{G}_{\mathcal{T}, \gamma}$ denote the subgroup of $G_{\mathcal{T}}$ spanned by $\mathcal{E}_\gamma$. Note that each $\hat{G}_{\mathcal{T}, \gamma}$ is $\mathcal{L}_{\infty \omega}$-definable, since $(\mathcal{T}_\gamma: \gamma < \lambda)$ is an antichain, and so $g \in \hat{G}_{\mathcal{T}, \gamma}$ if and only if $g = 0$ or else $\mathcal{T}_\gamma$ embeds into $\mathcal{T}^*_g$.

Let $\mathbf{F}(\overline{G}^\ell) = (G_{\mathcal{T}} \oplus G^\ell,  G_{\mathcal{T}, n, i}, H^0, H^1, \pi, \psi^\ell: n, i < \omega) \models \Omega^p_{\omega \times \omega \cup \{0, 1\}, \{0, 1\}}$, where $H^0 = \mathcal{T} \otimes \mathbb{Z}$, $H^1 = G^\ell$, and where $\psi^\ell: G_\mathcal{T} \to G^\ell$ satisfies:

\begin{itemize}
\item $\psi^\ell(t) = 0$ for all $t \in T$ not of height $1$,
\item For every $\gamma < \lambda$, $\psi \restriction_{\mathcal{E}_{\gamma}}: \mathcal{E}_{\gamma} \to G_\gamma^\ell$ is $\lambda$-to-one.
\end{itemize}

It is easy to check that this works. 

\end{proof}

Thus, to finish it suffices to verify the following:

\begin{lemma}
Suppose $\kappa(\omega)$ does not exist. Suppose $\alpha_* < \kappa(\omega)$. Then for some index set $\I $, there are $\overline{G}^0_*, \overline{G}^1_* \models \Omega_{\I, \{0\}}^p$, with $\overline{G}^0_* \sim_{\alpha_*} \overline{G}^1_*$ yet $\overline{G}^0_* \not \equiv_{\infty \omega} \overline{G}^1_*$.
\end{lemma}

Our idea is the following: given $\overline{G} = (G, G_i: i \in \I, \phi) \models \Omega_{\I, \{0\}}^p$, define $X^{\overline{G}} := G \backslash \bigcup_i G_i$ and define $\leq^{\overline{G}}$ to be the partial order of $X^{\overline{G}}$ given by: $a \leq^{\overline{G}} b$ if and only if $\phi^n(a) = b$ for some $n < \omega$, satisfying further that for all $m < n$, $\phi^m(a) \in X^{\overline{G}}$. Then we will arrange that $(X^{\overline{G}^0_*}, \leq^{\overline{G}^0_*})$  is ill-founded, but $(X^{\overline{G}^1_*}, \leq^{\overline{G}^1_*})$ is well-founded. It turns out we can make $\overline{G}^0_* \sim_{\alpha_*} \overline{G}^1_*$ without upsetting this.

We will be approximating $\overline{G}^0_*$ and $\overline{G}^1_*$ as a union of chains. To control the eventual behavior of $(X^{\overline{G}^i_*}, \leq^{\overline{G}^i_*})$, we will be defining upper bounds to the rank function at each stage. The following are the approximations we will be using:

\begin{definition}
Given an index set $\I$, let $\Gamma_{\I}$ denote all tuples $(\overline{G}, \mathcal{B}, \rho)$ where:

\begin{itemize}
\item $\overline{G} = (G, G_i, \phi: i \in \I) \models \Omega_{\I, \{0\}}^p$;
\item $G$ is free abelian (this is not redundant, since $\Omega_{\I, \{0\}}^p$ only asserts that $G \equiv_{\infty \omega} \oplus_\omega \mathbb{Z}$) and $\mathcal{B}$ is a basis of $G$;
\item $\phi: G \to G$;
\item $\rho: X^{\overline{G}} \to \mbox{ON} \cup \{\infty\}$ satisfies: for all $a, b \in X^{\overline{G}}$, if $\phi(b) = a$ and $\rho(b) < \infty$ then $\rho(a) < \rho(b)$. Hence $\rho(a) \geq \mbox{rnk}(a)$ where $\mbox{rnk}$ is the rank function for $(X^{\overline{G}}, \leq^{\overline{G}})$.
\item For all $a \in X$ and for all $n \in \mathbb{Z}$ nonzero, $\rho(a) = \rho(na)$.
\end{itemize}

When we write $\overline{G}, \overline{G}', \overline{G}^\ell$, etc., then we will always have $\overline{G} = (G, G_i, \phi: i \in I)$, $\overline{G}' = (G', G_i', \phi': i \in I), \overline{G}^\ell =  (G^\ell, G_i^{\ell}, \phi^\ell: i \in I),$ etc. 

\end{definition}
\begin{definition}
Suppose $\I, \I'$ are index sets with $\I \subseteq \I'$. Suppose $(\overline{G}, \mathcal{B}, \rho) \in \Gamma_{\I}$ and $(\overline{G}', \mathcal{B}', \rho') \in \Gamma_{\I'}$. Then say that $(\overline{G}', \mathcal{B}', \rho')$ extends $(\overline{G}, \mathcal{B}, \rho)$ if:

\begin{itemize}
\item $G \subseteq G'$ and $\mathcal{B} \subseteq \mathcal{B}'$;
\item For each $i \in \I$, $G'_i \cap G = G_i$;
\item For each $i \in \I' \backslash \I$, $G'_i \cap G = 0$;
\item $\phi' \restriction_{G_i} = \phi$;
\item $\rho' \restriction_{X^{\overline{G}}} = \rho$.
\end{itemize}

\end{definition}

The following lemma is immediate.

\begin{lemma}\label{UnionsOfChains}
Suppose $\delta < \lambda^+$ is a limit ordinal, $(\I_\gamma: \gamma < \delta)$ is an increasing chain of index sets, and $((\overline{G}^\gamma, \mathcal{B}^\gamma, \rho^\gamma): \gamma < \delta)$ is a sequence satisfying each $(\overline{G}^\gamma, \mathcal{B}^\gamma, \rho^\gamma) \in \Gamma_{\I_\gamma}$ and for $\gamma < \gamma'$, $(\overline{G}^{\gamma'}, \mathcal{B}^{\gamma'}, \rho^{\gamma'})$ extends $(\overline{G}^\gamma, \mathcal{B}^\gamma, \rho^\gamma)$. Then the natural union of the chain $(\overline{G}, \mathcal{B}, \rho)$ extends each $(\overline{G}^\gamma, \mathcal{B}^\gamma, \rho^\gamma)$. 
\end{lemma}

The final set of definitions describe the embeddings we will use to arrange $\overline{G}^0_* \sim_{\alpha_*} \overline{G}^1_*$.

\begin{definition}
If $(\overline{G}, \mathcal{B}, \rho) \in \Gamma_{\I}$, then say that $H$ is a basic subgroup of $G$ if $H$ is spanned by $H \cap \mathcal{B}$. By $\overline{G} \restriction_H$ we mean $(H, G_i \cap H, \phi \restriction_H: i \in I) \models \Omega_{\I, \{0\}}^p$. By $(\overline{G}, \mathcal{B}, \rho) \restriction_H$ we mean $(\overline{G} \restriction_H, \mathcal{B} \cap H, \rho \restriction_{X^{\overline{G} \restriction_H}})$.

Suppose $(\overline{G}, \mathcal{B}, \rho), (\overline{G}', \mathcal{B}', \rho') \in \Gamma_{\I}$. Then by a $-1$-embedding from $(\overline{G}, \mathcal{B}, \rho)$ into $(\overline{G}', \mathcal{B}', \rho')$, we mean a map $f$ where  $f: \overline{G} \leq \overline{G}'$ is an embedding and $f[\mathcal{B}] \subseteq \mathcal{B}'$. For an ordinal $\alpha \geq 0$, say that $f$ is an $\alpha$-embedding if additionally: $f[X^{\overline{G}}] \subseteq X^{\overline{G}'}$, and for all $a \in X^{\overline{G}}$, if $\rho(a) < \omega \cdot \alpha$, then $\rho(a) = \rho'(f(a))$.

For all $\alpha \geq -1$, say that $f$ is a partial $\alpha$-embeddding from $(\overline{G}, \mathcal{B}, \rho)$ into $(\overline{G}', \mathcal{B}', \rho')$ if for some basic subgroup $D$ of $\overline{G}$, $f$ is an an $\alpha$-embedding from $(\overline{G}, \mathcal{B}, \rho) \restriction_D$ to $(\overline{G}', \mathcal{B}', \rho')$.
\end{definition}

Finally, we describe the construction of $\overline{G}^0_*, \overline{G}^1_*$. We will build them as a union of chains. In the outer layer, we will construct, by induction on $n < \omega$, index sets $\I^n$, and, for each $\ell < 2$, $(\overline{G}^{\ell, n}, \mathcal{B}^{\ell, n}, \rho^{\ell, n}) \in \Gamma_{\I^n}$ with a privileged element $e^n \in G^{0, n}$ for $n > 0$, and for each $\ell < 2$ a set $\mathcal{F}^{\ell, n}$, satisfying various constraints. The goal is that $(e^n: n < \omega)$ will witness that $X^{\overline{G}^{0, n}}$ is ill-founded, and $\mathcal{F}^{\ell, n}$ will be a set of partial embeddings from $\overline{G}^{\ell, n}$ to $\overline{G}^{1-\ell, n}$, which will be used to arrange that $\overline{G}^0_* \sim_{\alpha_*} \overline{G}^1_*$. Formally, we need the following requirements:

\begin{enumerate}
\item For $n < m < \omega$, $(\overline{G}^{\ell, m}, \mathcal{B}^{\ell, m}, \rho^{\ell, m})$ extends $(\overline{G}^{\ell, n}, \mathcal{B}^{\ell, n}, \rho^{\ell, n})$; 
\item For each $n > 0$, $e^n \in X^{\overline{G}^{0, n}}$, and $\phi^{0, n+1}(e_{n+1}) = e_n$ (so necessarily each $\rho^{0, n}(e_n) = \infty$).
\item For all $a \in X^{\overline{G}^{1, n}}$, $\rho^{1, n}(a) < \infty$.
\item For all $n, \ell$, $(\phi^{\ell, n})^n = 0$ (i.e. $\phi^{\ell, n}$ iterated $n$-many times is $0$);
\item Each $\mathcal{F}^{\ell, n}$ is a set of tuples $(\alpha, D, R, f)$, where $-1 \leq \alpha \leq \alpha_*$, and $f$ is a partial $\alpha$-embedding from $(\overline{G}^{\ell, n}, \mathcal{B}^{\ell, n} , \rho^{\ell, n})$ to $(\overline{G}^{1-\ell, n}, \mathcal{B}^{1-\ell, n} , \rho^{1-\ell, n})$ with domain $D$ and range $R$;
\item For each $n < m$, and for each $\ell < 2$, $\mathcal{F}^{\ell, n} \subseteq \mathcal{F}^{\ell, m}$;
\item If $(\alpha, D, R, f) \in \mathcal{F}^{\ell, n}$ and $\alpha \geq 0$, then $(\alpha, R, D, f^{-1}) \in \mathcal{F}^{1-\ell, n}$ (in particular $f^{-1}$ is a partial $\alpha$-embedding);
\item Suppose $(\alpha, D, R, f) \in \mathcal{F}^{\ell, n}$, and suppose either $\beta < \alpha$ or else $\beta = -1$. Then for every $a \in G^{\ell, n+1}$, there is some $D' \supseteq D \cup \{a\}$, $R' \supseteq R$, and $f' \supseteq f$ such that $(\beta, D', R', f') \in \mathcal{F}^{\ell, n+1}$;
\item $G^{0, 0} = G^{1, 0} = 0$ (this determines each $(\overline{G}^{\ell, 0}, \mathcal{B}^{\ell, 0}, \rho^{\ell,0})$), and $(\alpha_*, 0, 0, 0) \in \mathcal{F}^{0, 0}$.

\end{enumerate}

Having done this, let $(\overline{G}^\ell_*, \mathcal{B}^\ell_*, \rho^\ell_*)$ be the union of the chain $((\overline{G}^{\ell, m}, \mathcal{B}^{\ell, m}, \rho^{\ell, m}): m < \omega)$, as promised by Lemma~\ref{UnionsOfChains}. Then $\overline{G}^{0}_* \not \equiv_{\infty \omega} \overline{G}^1_*$, since $(X^{\overline{G}^0_*}, \leq^{\overline{G}^0_*})$ is ill-founded (by condition (2)) while $(X^{\overline{G}^1_*}, \leq^{\overline{G}^1_*})$ is well-founded (by condition (3). On the other hand, it is clear that for all $n< \omega$, for all $(\alpha, D, R, f) \in \mathcal{F}^{\ell, n}$ with $\alpha \geq 0$, and for all finite tuples $\overline{a} \in D$, we have $(\overline{G}^{\ell}_*, \overline{a}) \sim_\alpha (\overline{G}^{1-\ell}_*, f(\overline{a}))$ (by condition (8)). Thus $\overline{G}^{0}_* \sim_{\alpha_*} \overline{G}^1_*$.

So it remains to show this construction is possible. This will mostly be achieved by the following two lemmas, which will allow us to handle the key condition (8) without disturbing any of the other hypotheses:

\begin{lemma}\label{ExtendLemma1}
Suppose $(\overline{G}^{\ell}, \mathcal{B}^{\ell}, \rho^{\ell}) \in \Gamma_{\I}$ for each $\ell < 2$. Suppose $f$ is a partial $ -1$-embedding from $(\overline{G}^0, \mathcal{B}^0, \rho^0)$ to $(\overline{G}^1, \mathcal{B}^1, \rho^1)$. Finally, suppose each $(\phi^i)^{n+1} = 0$. Then we can find an index set $\I'$, and an extension $(\overline{G}^{2}, \mathcal{B}^{2}, \rho^{2})$ of $(\overline{G}^{1}, \mathcal{B}^{1}, \rho^{1})$ in $\Gamma_{\mathcal{I}'}$, such that $X^{\overline{G}_2} = X^{\overline{G}_1}$, and $f$ extends to a $-1$-embedding $h$ from $(\overline{G}^0, \mathcal{B}^0, \rho^0)$ to $(\overline{G}^2, \mathcal{B}^2, \rho^2)$, and finally $(\phi^2)^{n+1} = 0$.
\end{lemma}
\begin{proof}
Let $D$ be the domain of $f$ and let $R$ be its range. Recall that we require $D$ and $R$ to be basic subgroup of $G$, that is, $\mathcal{B} \cap D$ spans $D$. Let $\I' \supseteq \I$ be large enough.

Write $\mathcal{A} = \mathcal{B}^0 \backslash (\mathcal{B} \cap D) $. Let $G^2 = G^1 \times \oplus_{\mathcal{A}} \mathbb{Z}$. Write $H = 0 \times \oplus_{\mathcal{A}} \mathbb{Z}$, and let $g: \mbox{span}_{G^0}(\mathcal{A}) \cong H$ be the natural isomorphism. Let $\mathcal{B}^2$ be $\mathcal{B}^1 \cup g[\mathcal{A}]$.  Define $h: G^0 \to G^2$ via $h \restriction_{D} = f$ and $h \restriction_{\mbox{span}(\mathcal{A})}= g$. 

Define $\phi^2: G^2 \to G^2$ via: $\phi^2 \restriction_{G^1} = \phi^1$, and $\phi^2 \restriction_{H} = g \circ \phi^0 \circ g^{-1}$. For each $i \in \I$, let $G^2_i = G^1_i$.

Let $G^2_i: i \in \I' \backslash \I$ enumerate all singly generated pure subgroups of $G^2$ which are not contained in $G^1$. Note then that $X^{\overline{G}^2} = X^{\overline{G}^1}$ so we must let $\rho^2 = \rho^1$ and then clearly we are done.

\end{proof}

\begin{lemma}\label{ExtendLemma2}
Suppose $(\overline{G}^{\ell}, \mathcal{B}^{\ell}, \rho^{\ell}) \in \Gamma_{\I}$ for each $\ell < 2$. Suppose $0 \leq \beta < \alpha$, and $f$ is a partial $\alpha$-embedding from $(\overline{G}^0, \mathcal{B}^0, \rho^0)$ to $(\overline{G}^1, \mathcal{B}^1, \rho^1)$ such that $f^{-1}$ is also a partial $\alpha$-embedding. Finally, suppose each $(\phi^i)^{n+1} = 0$. Then we can find an index set $\I'$, and an extension $(\overline{G}^{2}, \mathcal{B}^{2}, \rho^{2})$ of $(\overline{G}^{1}, \mathcal{B}^{1}, \rho^{1})$ in $\Gamma_{\mathcal{I}'}$, such that:

\begin{itemize}
\item $f$ extends to an $\beta$-embedding $h$ from $(\overline{G}^0, \mathcal{B}^0, \rho^0)$ to $(\overline{G}^2, \mathcal{B}^2, \rho^2)$;
\item $h^{-1}$ is a partial $\beta$-embedding from $(\overline{G}^2, \mathcal{B}^2, \rho^2)$ to $(\overline{G}^0, \mathcal{B}^0, \rho^0)$;
 \item For all $a \in X^{\overline{G}^2} \backslash X^{\overline{G}^1}$, $\rho^2(a) < \omega \cdot \alpha$;
 \item $(\phi^2)^{n+1} = 0$.
 \end{itemize}
\end{lemma}
\begin{proof}
Let $D$ be the domain of $f$ and let $R$ be its range. Let $\I' \supseteq \I$ be large enough.

Write $\mathcal{B}^0 = (\mathcal{B} \cap D) \cup \mathcal{A}$. Let $G^2 = G^1 \times \oplus_{\mathcal{A}} \mathbb{Z}$. Write $H = 0 \times \oplus_{\mathcal{A}} \mathbb{Z}$, and let $g: \mbox{span}_{G^0}(\mathcal{A}) \cong H$ be the natural isomorphism. Let $\mathcal{B}^2$ be $\mathcal{B}^1 \cup g[\mathcal{A}]$.  Define $h: G^0 \to G^2$ via $h \restriction_{D} = f$ and $h \restriction_{\mbox{span}(\mathcal{A})}= g$. 

Define $\phi^2: G^2 \to G^2$ via: $\phi^2 \restriction_{G^1} = \phi^1$, and $\phi^2 \restriction_{H} = g \circ \phi^0 \circ g^{-1}$. For each $i \in \I$, let $G^2_i = G^1_i$. It remains to define $G^2_i$ for $i \in \I' \backslash \I$, and then to define $\rho^2$.

Let $G^2_i: i \in \I' \backslash \I$ enumerate all singly generated pure subgroups of $G^2$ which are not contained in $G^1$ and which are not contained in $R + H$. Note then that $X^{\overline{G}^2} = X^{\overline{G}^1} \cup h[X^{\overline{G}^0}]$. We define $\rho^2$ as follows: suppose $a \in X^{\overline{G}^2}$. If $a \in X^{\overline{G}^1}$ then we must let $\rho^2(a) = \rho^1(a)$. Suppose instead $a \in h[X^{\overline{G}^0}] \backslash X^{\overline{G}^1}$; write $a = h(a')$. If $\rho^0(a) < \omega \cdot \beta$ then let $\rho^2(a) = \rho^0(a)$. Otherwise, let $k$ be largest such that there is $c' \in X^{\overline{G}^0}$ such that $(\phi^0)^k(c') = a'$, and for all $k' < k$, $(\phi^0)^{k'}(c') \in X^{\overline{G}^0}$, and finally $\rho^0(c') \geq \omega \cdot \beta$; let $\rho^2(a) = \omega \cdot \beta + k$. Note that $k \leq n$ since $(\phi_0)^{n+1} = 0$.

Now I claim this works. First of all:

\vspace{1 mm}

\noindent \textbf{Claim.} Suppose $a \in h[X^{\overline{G}^0}] \backslash X^{\overline{G}^1}$; write $h(a') =a$. Then $\rho^0(a') \geq \rho^2(a)$. 

\begin{proof}This is immediate if $\rho^0(a') < \omega \cdot \beta$, so suppose instead $\rho^0(a') \geq \omega \cdot \beta$; let $c', k$ be as in the definition of $\rho^2(a)$. Then $\rho^0(a') = \rho^0((\phi_0)^k(c')) \geq \rho^0(c') + k \geq \omega \cdot \beta +k = \rho^2(a)$. 
\end{proof}

We show $(\overline{G}^2, \mathcal{B}^2, \rho^2) \in \Gamma_{\I'}$. We must check that for all $a, b \in X^{\overline{G}^2}$ with $\phi^2(b) = a$ and with $\rho^2(a) < \infty$, we have that $\rho^2(b) < \rho^2(a)$. If $b \in X^{\overline{G}^1}$, then $a \in X^{\overline{G}^1}$ and this is clear. Suppose $b \in h[X^{\overline{G}^0}] \backslash X^{\overline{G}^1}$, and $a \in X^{\overline{G}^1}$;  note that $a \in f[X^{\overline{G}^0}] \subseteq h[X^{\overline{G}^0}]$; write $a = f(a')$ and write $b = h(b')$. We consider two further subcases. If $\rho^0(a') = \rho^1(a)$, then $\rho^2(a) = \rho^0(a') > \rho^0(b') \geq \rho^2(b)$, using the claim. If $\rho^0(a') \not= \rho^1(a)$, then since $f, f^{-1}$ are both  $\alpha$-embeddings we must have $\rho^0(a'), \rho^1(a) \geq \omega \cdot \alpha$. Hence $\rho^2(a) = \rho^1(a) \geq \omega \cdot \alpha > \omega \cdot \beta + n \geq \rho^2(b)$. Finally, suppose both $a, b \in h[X^{\overline{G}^0}] \backslash X^{\overline{G}^1}$. Write $a = h(a')$, $b = h(b')$. If $\rho^0(a') < \omega \cdot \beta$ then $\rho^2(a) = \rho^0(a') > \rho^0(b') \geq \rho^2(b)$. If $\rho^0(a') \geq \omega \cdot \beta$ and $\rho^0(b') < \omega \cdot \beta$, then $\rho^2(a) \geq \omega \cdot \beta > \rho^0(b') = \rho^2(b)$. Finally, if $\rho^0(a')$ and $\rho^0(b')$ are both $\geq \omega \cdot \beta$, then let $k$ be as in the definition of $\rho^2(b)$, i.e. so that $\rho^2(b) = \omega \cdot \beta + k$; clearly then $\rho^2(a) \geq \omega \cdot \beta + (k+1)$.

To finish, it is clear that for all $a' \in \overline{G}^0$, if either $\rho^0(a') < \omega \cdot \beta$ or else $\rho^2(h(a')) < \omega \cdot \beta$, then $\rho^0(a') = \rho^2(h(a'))$; hence $h$ is a $\beta$-embedding and $h^{-1}$ is a partial $\beta$-embedding.
\end{proof}

Now, suppose we are given $(\overline{G}^{\ell, n}, \mathcal{B}^{\ell, n}, \rho^{\ell, n}),$ $\mathcal{F}^{\ell, n}$, and $e^n$ satisfying (1) through (9). We explain how to get $(\overline{G}^{\ell, n+1}, \mathcal{B}^{\ell, n+1}, \rho^{\ell, n+1}),$ $\mathcal{F}^{\ell, n+1}$, and $e^{n+1}$.

Define $G^0 = G^{0, n} \times \mathbb{Z}$, let $e^{n+1} = (0, 1) \in G^0$. Let $\I \supseteq \I^n$ be sufficiently large. For each $i \in \I^{n}$ let $G_i^0 = G_i^{0, n}$. Choose $(G_i^0: i \in \I \backslash \I^n)$ so as to enumerate the singly-generated pure subgroups of $G^0$ which are not contained in $G^{0, n}$ and which do not contain $e^{n+1}$. Define $\phi^{0}$ via $\phi^{0} \restriction_{G^{0, n}} = \phi^{0, n}$ and $\phi^{0}(e^{n+1}) = e^n$ (or, if $n = 0$ then let $\phi^0(e_1) = 0$). We have defined $\overline{G}^{0} \models \Omega_{\I, \{0\}}^p$, an extension of $\overline{G}^{0, n}$. Note that $X^{\overline{G}^{0}} = X^{\overline{G}^{n, 0}} \cup \{m e^{n+1}: m \in \mathbb{Z}, m \not= 0\}$. Let $\mathcal{B}^{0} = \mathcal{B}^{0, n} \cup \{e^{n+1}\}$, and define each $\rho^{0}(m e^{n+1}) = \infty$.

Define $G^1 = G^{1, n}$; for each $i \in \I^n$, let $G^1_i = G^1$, and for each $i \in \I \backslash \I^n$, and let $G^1_i = 0$; let $\phi^1 = \phi^{1, n}$. Finally, let $\mathcal{F}^{\ell} = \mathcal{F}^{\ell, n}$ for each $\ell < 2$.

The only thing left to do is arrange (8) to hold. For this, apply Lemmas~\ref{ExtendLemma1} and \ref{ExtendLemma2} repeatedly, using Lemma~\ref{UnionsOfChains} at limit stages.

This concludes the proof of Theorem \ref{SBFailureTheorem}, and hence of Theorem \ref{SBFailureThm}.

\end{document}